\documentclass[12pt]{article}
\usepackage{amsmath}
\usepackage{amssymb}
\usepackage{amsthm}
\usepackage{graphicx}
\usepackage{euscript}
\usepackage{hyperref}

\title{Transseries: Composition,\\
Recursion, and Convergence}

\author{G. A. Edgar}

\date{\today}

\theoremstyle{plain}
\newtheorem{pr}{Proposition}
\newtheorem{thm}[pr]{Theorem}
\newtheorem{co}[pr]{Corollary}
\newtheorem{lem}[pr]{Lemma}

\theoremstyle{remark}
\newtheorem{re}[pr]{Remark}
\newtheorem{de}[pr]{Definition}
\newtheorem{no}[pr]{Notation}
\newtheorem{ex}[pr]{Example}

\newtheorem{qu}[pr]{Question}

\numberwithin{pr}{section}

\newcommand{\Def}[1]{\textbf{\itshape #1}} 

\newcommand{\ECadd}{Prop.~3.24}
\newcommand{\Ewellprod}{Prop.~3.27}
\newcommand{\ECmult}{Prop.~3.29}
\newcommand{\Easympt}{Def.~3.45}
\newcommand{\Emultcontin}{Prop.~3.48}

\newcommand{\Elogfreepower}{Prop.~3.71}
\newcommand{\Eheightwins}{Prop.~3.72}
\newcommand{\Ederivexist}{Prop.~3.76}

\newcommand{\EWKB}{Prop.~3.82}
\newcommand{\Ecompexist}{Prop.~3.95}
\newcommand{\EcomponN}{Prop.~3.98}
\newcommand{\Ecompcontin}{Prop.~3.99}

\newcommand{\Edominateprop}{Prop.~4.17}
\newcommand{\Eintegral}{Prop.~4.29}

\newcommand{\Wgeomconv}{Def.~3.15}
\newcommand{\Wderivconvergence}{Prop.~4.7}
\newcommand{\Wtsupp}{Def.~7.1}
\newcommand{\Winv}{Sec.~8}
\newcommand{\Wfuppermonoc}{Rem.~9.3}

\renewcommand{\phi}{\varphi}
\renewcommand{\epsilon}{\varepsilon}
\renewcommand{\emptyset}{\varnothing}
\newcommand{\takes}{\colon}
\newcommand{\fgt}{\succ}
\newcommand{\fst}{\prec}
\newcommand{\fe}{\asymp}
\newcommand{\fgteq}{\succcurlyeq}
\newcommand{\fsteq}{\preccurlyeq}

\newcommand{\SET}[2]{ \left\{\, {#1} : {#2} \,\right\} }
\newcommand{\R}{\mathbb R}

\newcommand{\N}{\mathbb N}
\newcommand{\Z}{\mathbb Z}

\newcommand{\G}{\mathfrak G}

\newcommand{\GRID}{\mathfrak J}
\newcommand{\WW}{\mathfrak W}
\renewcommand{\AA}{\mathfrak A}
\newcommand{\BB}{\mathfrak B}
\newcommand{\FC}{\mathfrak C}
\newcommand{\M}{\mathfrak M}
\newcommand{\T}{\mathbb T}
\newcommand{\RR}{\EuScript R}
\renewcommand{\SS}{\EuScript S}
\newcommand{\LP}{\EuScript P}
\renewcommand{\P}{\EuScript P}

\newcommand{\SA}{\EuScript A}
\newcommand{\SU}{\EuScript U}

\renewcommand{\o}{\mathrm o}
\renewcommand{\O}{\mathrm O}
\newcommand{\omu}{\o_\ebmu}
\newcommand{\Omu}{\O_\ebmu}

\newcommand{\muto}{\overset{\ebmu}{\longrightarrow}}

\newcommand{\ac}{\,\raisebox{0.4ex}{$\underset{{}^{{}^\mathrm C}}{\longrightarrow}$}\,}
\newcommand{\aw}{\,\raisebox{0.4ex}{$\underset{{}^{{}^\mathrm W}}{\longrightarrow}$}\,}
\newcommand{\ah}{\,\raisebox{0.4ex}{$\underset{{}^{{}^\mathrm H}}{\longrightarrow}$}\,}

\newcommand{\bm}{\mathbf m}

\newcommand{\fa}{\mathfrak a}
\newcommand{\fb}{\mathfrak b}

\newcommand{\g}{\mathfrak g}
\newcommand{\m}{\mathfrak m}
\newcommand{\n}{\mathfrak n}
\newcommand{\e}{\mathfrak e}

\renewcommand{\l}{\mathfrak l}
\newcommand{\A}{\mathbf A}
\newcommand{\B}{\mathbf B}
\newcommand{\CC}{\mathbf C}
\newcommand{\DD}{\mathbf D}

\newcommand{\Gsmall}{\G^{\mathrm{small}} }
\newcommand{\Msmall}{\M^{\mathrm{small}} }
\newcommand{\Glarge}{\G^{\mathrm{large}} }
\newcommand{\Wlarge}{\WW^{\mathrm{large}} }
\newcommand{\Wpure}{\WW^{\mathrm{pure}} }
\newcommand{\supp}{\operatorname{supp}}
\newcommand{\lsupp}{\operatorname{lsupp}}
\newcommand{\tsupp}{\operatorname{tsupp}}
\renewcommand{\mag}{\operatorname{mag}}

\newcommand{\dom}{\operatorname{dom}}

\newcommand{\expo}{\operatorname{expo}}

\newcommand{\bmu}{{\boldsymbol{\mu}}}
\newcommand{\ebmu}{{\boldsymbol{\mu}}}
\newcommand{\tbmu}{{\widetilde{\bmu}}}
\newcommand{\tebmu}{{\tilde{\ebmu}}}
\newcommand{\ba}{{\boldsymbol{\alpha}}}

\newlength{\uuu}
\newcommand{\lbb}{\begin{picture}(8,8)(-2,2)
	\put(0,0){\line(0,1){9}}
	\put(2,0){\line(0,1){9}}
	\put(0,0){\line(1,0){5}}
	\put(0,9){\line(1,0){5}}
	\end{picture}}
\newcommand{\rbb}{\begin{picture}(7,8)(0,2)
	\put(3,0){\line(0,1){9}}
	\put(5,0){\line(0,1){9}}
	\put(0,0){\line(1,0){5}}
	\put(0,9){\line(1,0){5}}
	\end{picture}}
\newcommand{\lbbb}{\begin{picture}(10,8)(-2,2)
	\put(0,0){\line(0,1){9}}
	\put(2,0){\line(0,1){9}}
	\put(4,0){\line(0,1){9}}
	\put(0,0){\line(1,0){7}}
	\put(0,9){\line(1,0){7}}
	\end{picture}}
\newcommand{\rbbb}{\begin{picture}(9,8)(0,2)
	\put(3,0){\line(0,1){9}}
	\put(5,0){\line(0,1){9}}
	\put(7,0){\line(0,1){9}}
	\put(0,0){\line(1,0){7}}
	\put(0,9){\line(1,0){7}}
	\end{picture}}

\begin{document} 
\maketitle
\setcounter{tocdepth}{4}

\abstract{Additional remarks and questions for transseries.
In particular: properties of composition for transseries;
the recursive nature of the construction of $\R\lbbb x \rbbb$;
modes of convergence for transseries.
There are, at this stage, questions and missing proofs
in the development.
}

\tableofcontents

\section{Introduction}
Most of the calculations done with transseries are easy, once
the basic framework is established.  But that may not be the
case for composition of transseres.  Here I will discuss
a few of the interesting features of composition.

The ordered differential field $\T = \R\lbbb x \rbbb = \R\lbb \G \rbb$ of
(real grid-based) transseries is completely explained in
my recent expository introduction \cite{edgar}.  A related
paper is \cite{edgarw}.  Other
sources for the definitions are:
\cite{asch}, \cite{costintop}, \cite{DMM}, \cite{hoeven},
\cite{kuhlmann}.
I will generally follow the notation from \cite{edgar}.  
Van der Hoeven \cite{hoeven}
sometimes calls $\T$ \Def{the transline}.

So $\T$ is the set of all grid-based real formal linear
combinations of monomials from $\G$, while $\G$
is the set of all $e^L$ for $L \in \T$ purely large.
(Because of logarithms, there is no need to write
separately two factors as $x^b e^L$.)

\begin{no}
For transseries $A$, we already use
exponents $A^n$ for multiplicative powers,
and parentheses $A^{(n)}$ for derivatives.  Therefore
let us use square brackets
$A^{[n]}$ for compositional powers.
In particular, we will write $A^{[-1]}$ for the compositional
inverse.  Thus, for example, $\exp_n = \exp^{[n]} = \log^{[-n]}$.

Write $\l_n$ for $\log_n$ if $n>0$; write $\l_0 = x$;
write $\l_n = \exp_{-n}$ if $n<0$.
\end{no}

Recall \cite[\ECadd\ \& \ECmult]{edgar}
two canonical decompositions for a transseries:

\begin{pr}[Canonical Additive Decomposition]\label{C_add}
Every $A \in \R\lbb\G\rbb$ may be written uniquely in the form
$A = L + c + V$, where $L$ is purely large,
$c$ is a constant, and $V$ is small.
\end{pr}

\begin{pr}[Canonical Multiplicative Decomposition]\label{C_mult}
Every nonzero transseries $A \in \R\lbb\G\rbb$ may be
written uniquely in the form
$A = a \cdot\g \cdot(1+U)$ where $a$ is nonzero real,
$\g \in \G$, and $U$ is small.
\end{pr}

\begin{no}
Little-$\o$ and big-$\O$.  For $A \ne 0$ we define sets,
$$
	\o(A) := \SET{T \in \T}{T \fst A},
	\qquad
	\O(A) := \SET{T \in \T}{T \fsteq A}.
$$
These are used especially when $A$ is a monomial, but
$\o(A) = \o(\mag A)$.
Conventionally, we write $T = U + \o(A)$ when we mean
$T \in U + \o(A)$ or $T-U \fst A$.
\end{no}

\begin{no}
For use with a finite ratio set $\bmu \subset \Gsmall$,
we define
$$
	\omu(A) := \SET{T \in \T}{T \fst^\ebmu A},
	\qquad
	\Omu(A) := \SET{T \in \T}{T \fsteq^\ebmu A}.
$$
This time monomials do not suffice: if $\bmu = \{x^{-1},e^{-x}\}$,
then $\omu(x^{-1}+e^{-x}) \ne \omu(x^{-1})$.
\end{no}

\begin{re}\label{relationship}
Note the simple relationship between $<$ and $\fst$:
Define $|T| = T$ if $T \ge 0$, $|T| = -T$ if $T < 0$.
Then
\begin{align*}
	U \fst V & \Longleftrightarrow
	|U| < k|V| \text{ for all $k \in \R, k>0$},
\\
	U \fsteq V & \Longleftrightarrow
	|U| < k|V| \text{ for some $k \in \R, k>0$},
\\
	U \fe V & \Longleftrightarrow
	\frac{1}{k} < \left|\frac{U}{V}\right| <
	k  \text{ for some $k \in \R, k>1$},
\\
	U \sim V & \Longleftrightarrow
	\frac{1}{k} < \frac{U}{V} <
	k  \text{ for all $k \in \R, k>1$}.
\end{align*}
\end{re}
\noindent
The reason we can do this is the following interesting property:
if $1/k < T < k$ for some $k \in \R$, $k>1$,
then there is $c \in \R$, $c>0$, with $T \sim c$.

\begin{re}
Worth noting:  If $0 < A \le B$, then $A \fsteq B$.
If $0 > A \ge B$, then $A \fsteq B$.  If $A>0, B>0, A \fst B$,
then $A < B$.  If $A<0, B<0, A \fst B$, then $A > B$.
\end{re}

\section{Well-Based Transseries}
Besides the grid-based transseries as found in \cite{edgar},
we may also refer to the well-based version
as found, for example in \cite{DMM} or \cite{kuhlmann}.

\begin{de}\label{beyond}
For an ordered abelian group $\M$, let $\R[[\M]]$
be the set of Hahn series with support which is well ordered
(according to the reverse of $\fgt$).  Begin with group
$\WW_{0} = \SET{x^a}{a\in \R}$
and field $\T_{0} = \R[[\WW_{0}]]$.
Assuming field $\T_{N} = \R[[\WW_{N}]]$ has been defined,
let
$$
	\WW_{N+1} = \SET{x^b e^L}{L \in \T_{N} \text{ is purely large}}
$$
and $\T_{N+1} = \R[[\WW_{N+1}]]$.  Then
$$
	\WW_{\bullet} = \bigcup_{N=0}^\infty \WW_{N},\qquad
	\T_\bullet = \bigcup_{N=0}^\infty \T_{N}.
$$
Now as before,
\begin{align*}
	\WW_{\bullet,M} &= \SET{\g\circ\log_M}{\g \in \WW_{\bullet}},\qquad
	\T_{\bullet,M} = \SET{T\circ \log_M}{T \in \T_{\bullet}},
	\\
	\WW_{\bullet,\bullet} &= \bigcup_{M=0}^\infty \WW_{\bullet,M},\qquad
	\T_{\bullet,\bullet} = \bigcup_{M=0}^\infty \T_{\bullet,M}.
\end{align*}
A difference from the grid-based case:
$\T_{\bullet,\bullet} \ne \R[[\WW_{\bullet,\bullet}]]$.
The domain of $\exp$ is $\T_{\bullet,\bullet}$ and not
all of $\R[[\WW_{\bullet,\bullet}]]$.

Then $\T = \T_{\bullet,\bullet}$ is what I will mean here by ``well based''
transseries.  This is the system found in \cite{DMM}, for example.
This system and others are explored in \cite{kuhlmann}.
\end{de}

We have used letter Fraktur G ($\G$) for ``grid'' and letter
Fraktur W ($\WW$) for ``well''.  Notation $\T$ is used
for both, perhaps that will be confusing?  It is intended
that what I say here can usually apply to either case.

Here is one of the results that the well-based theory depends on.
(It is required, for example, to show that $T^{-1}$
has well-ordered support.)
I am putting it here because of its tricky proof.
The result is attributed to Higman, with this proof due
to Nash-Williams.

\begin{pr}\label{womonoid}
Let $\M$ be a totally ordered abelian group.  Let
$\BB\subseteq \Msmall$ be a set of small elements.
Write $\BB^*$ for the monoid generated by $\BB$.
If $\BB$ is well ordered (for the reverse of $\fgt$),
then $\BB^*$ is also well ordered.
\end{pr}
\begin{proof}
Write $\BB_n$ for the set of all products of $n$ elements
of $\BB$.  Thus:  $\BB_0 = \{1\}$, $\BB_1 = \BB$,
$\BB^* = \bigcup_{n=0}^\infty \BB_n$.  If $\g \in \BB^*$,
define the \Def{length} of $\g$ as
$$
	l(\g) = \min\SET{n}{\g \in \BB_n} .
$$
Since $\M$ is totally ordered, these are equivalent:

(i) $\BB$ is well ordered (every nonempty subset has
a greatest element),

(ii) any infinite sequence in $\BB$ has a nonincreasing subsequence,

(iii) there is no infinite strictly increasing sequence in $\BB$.

\noindent
We assume $\BB$ is well ordered, so it has all three properties.
We claim $\BB^*$ is well ordered.

Suppose (for purposes of contradiction) that there is
an infinite strictly increasing sequence in $\BB^*$.
Among all infinite strictly increasing sequences in $\BB^*$,
let $l_1$ be the minimum length of the first term.  Choose
$\n_1$ that has length $l_1$ and is the first term of
an infinite strictly increasing sequence in $\BB^*$.
Recursively, suppose that finite sequence
$\n_1 \fst \n_2 \fst \cdots \fst \n_k$ has been
chosen so that it is the beginning of some infinite
strictly increasing sequence in $\BB^*$.  Among all infinite
strictly increasing sequences in $\BB^*$ beginning
with $\n_1,\cdots,\n_k$, let $l_{k+1}$ be the minimum
length of the $(k+1)$st term.  Choose $\n_{k+1}$
of length $l_{k+1}$ such that there is an infinite strictly
increasing sequence in $\BB^*$ beginning
$\n_1,\cdots,\n_k,\n_{k+1}$.  This completes a recursive definition
of an infinite strictly increasing sequence $(\n_k)$ in $\BB^*$.

Now because all elements of $\BB$ are small and this sequence
is strictly increasing, $\n_k \ne 1$.  For each $k$, choose
a way to write $\n_k$ as a product of $l_k$ elements of
$\BB$, then let $\fb_k \in \BB$ be least of the factors.
So $\n_k = \fb_k \m_k$.  Now $(\fb_k)$ is an infinite
sequence in $\BB$, so there is a subsequence $(\fb_{k_j})$
with $\fb_{k_1} \fgteq \fb_{k_2} \fgteq \cdots$.
So
$$
	\m_{k_j} = \frac{\n_{k_j}}{\fb_{k_j}} \fst
	\frac{\n_{k_{j+1}}}{\fb_{k_j}} \fsteq
	\frac{\n_{k_{j+1}}}{\fb_{k_{j+1}}} = \m_{k_{j+1}}
$$
and (if $k_1 > 1$)
$$
	\n_{k_1-1} \fst \n_{k_1} \fsteq
	\frac{\n_{k_1}}{\fb_{k_1}} = \m_{k_1} .
$$
So $\n_1 \fst n_2 \fst \cdots \fst n_{k_1-1} \fst \m_{k_1} \fst
\m_{k_2} \fst \m_{k_3} \fst \cdots$ is an infinite strictly
increasing sequence in $\BB^*$.  But it begins with
$\n_1,\cdots,\n_{k_1-1}$ and $l(\m_{k_1}) = l_{k_1}-1$,
contradicting the minimality of $l_{k_1}$.  This contradiction shows
that there is, in fact, no infinite strictly increasing
seuqence in $\BB^*$.  So $\BB^*$ is well ordered.
\end{proof}

\begin{no}
For $N \in \N$, $N \ge 1$,
write
$$
	\Wpure_N = \SET{e^L}{L \text{ purely large, }
	\supp L \subset \WW_{N-1}\setminus\WW_{N-2}} ,
$$
$\Wpure_0 = \WW_0$, $\WW_{-1} = \{1\}$.
\end{no}

Of course the sets $\Wpure_N$ are subgroups
of $\WW_\bullet$.  Any $\g \in \WW_N$
can be written uniquely as $\g = \fa\fb$ with
$\fa \in \WW_{N-1}$ and $\fb \in \Wpure_N$.
Group $\WW_N$ is the direct product of subgroups:
$$
	\WW_N = \Wpure_0\cdot\Wpure_1\cdots
	\Wpure_{N-1}\cdot\Wpure_N .
$$

A set $\AA \subset \WW_N$ is decomposed as
\begin{equation*}
	\AA = \SET{\fa \fb}{\fb \in \BB, \fa \in \AA_\fb},
	\tag{$*$}
\end{equation*}
where $\BB \subset \Wpure_N$,
and for each $\fb \in \BB$, the set $\AA_\fb \subset \WW_{N-1}$.
The ordering in $\AA$ is lexicographic:
$$
	\fa_1\fb_1 \fst \fa_2\fb_2
	\quad\Longleftrightarrow\quad
	\fb_1 \fst \fb_2 \text{ or } \{
	\fb_1 = \fb_1 \text{ and } \fa_1 \fst \fa_2 \}.
$$
So the set $\AA$ is well ordered if and only if
set $\BB$ and all sets $\AA_\fb$ are well ordered.

The lexicographic ordering is the ``height wins'' rule:

\begin{pr}\label{Wheightwins}
Let $N \in \N$, $N \ge 1$.  If $\g \in \WW_N\setminus\WW_{N-1}$
and $\supp T \subset \WW_{N-1}$, then:
$T \fst \g$ if $\g \fgt 1$ and
$T \fgt \g$ if $\g \fst 1$.
\end{pr}

\subsection*{Decomposition of Sets}
I include here a few more uses of the decomposition
($*$).  Skip to Section~\ref{RecStr} if you are
primarily interested in the grid-based version
of the theory.

Write $\m^\dagger = \m'/\m$ for the logarithmic derivative.
In particular, if $\m = e^L \in \Wpure_N$, $N \ge 2$, then
$\m^\dagger = L'$ is supported in $\Wlarge_{N-1}\setminus\WW_{N-2}$,
and if $\m = e^L \in \Wpure_1$, then $\m^\dagger = L'$ is
supported in $\WW_0$.

The existence of the derivative for transseries is stated
like this:  If $T = \sum_{\g \in \AA} c_\g \g$,
then $T' = \sum_{\g\in\AA} c_\g \g'$.
Let us consider it more carefully.

\begin{thm}\label{wellderiv}
Let $\AA \subseteq \WW_{N,M}$ be well ordered, and let
$T = \sum_{\g \in \AA} c_\g \g$ in $\T_{N,M}$
have support $\AA$.  Then
{\rm(i)}~the family $\SET{\supp(\g')}{\g \in \AA}$ is point-finite;
{\rm(ii)}~$\bigcup_{\g\in\AA}\supp(\g')$ is well ordered;
{\rm(iii)}~$\sum_{\g\in\AA} c_\g \g'$ exists in
$\T_{\bullet,\bullet}$.
\end{thm}

This is proved in stages.

\begin{pr}\label{deriv0}
Let $\AA \subseteq \WW_{0}$ be well ordered, and let
$T = \sum_{\g \in \AA} c_\g \g$ have support $\AA$.  Then
{\rm(i)}~the family $\SET{\supp(\g')}{\g \in \AA}$ is point-finite;
{\rm(ii)}~$\bigcup_{\g\in\AA}\supp(\g')$ is well ordered;
{\rm(iii)}~$\sum_{\g\in\AA} c_\g \g'$ exists in
$\T_0$.
\end{pr}
\begin{proof}
Since $(x^b)' = bx^{b-1}$, the family $\SET{\supp(\g')}{\g \in \AA}$
is disjoint.  Then
$$
	\bigcup_{\g\in\AA}\supp(\g') \subseteq x^{-1}\AA,
$$
so it is well ordered. (iii)~follows from (i) and (ii).
\end{proof}

\begin{pr}\label{derivN}
Let $\AA \subseteq \WW_{N}$ be well ordered, and let
$T = \sum_{\g \in \AA} c_\g \g$ have support $\AA$.  Then
{\rm(i)}~the family $\SET{\supp(\g')}{\g \in \AA}$ is point-finite;
{\rm(ii)}~$\bigcup_{\g\in\AA}\supp(\g')$ is well ordered;
{\rm(iii)}~$\sum_{\g\in\AA} c_\g \g'$ exists in
$\T_N$.
\end{pr}
\begin{proof}
This will be proved by induction on $N$.  The case $N=0$
is Proposition~\ref{deriv0}.  Now let $N\ge 1$ and assume the result
holds for smaller values.  Decompose $\AA$ as usual:
\begin{equation*}
	\AA = \SET{\fa \fb}{\fb \in \BB, \fa \in \AA_\fb},
	\tag{$*$}
\end{equation*}
where $\BB \subset \Wpure_N$ is well ordered,
and for each $\fb \in \BB$, the set $\AA_\fb \subset \WW_{N-1}$
is well ordered.  Now if $\g = \fa\fb \in \AA$,
$\fb \in \Wpure_N$, $\fa \in \WW_{N-1}$,
then $\g' = (\fa' + \fa\fb^\dagger)\fb$ and
$\supp(\fa' + \fa\fb^\dagger) \subset \G_{N-1}$.

(i) Let $\m \in \WW$ belong to some $\supp(\g')$.
It could be that $\m \in\supp(\fa')\fb$, $\fb \in \BB$,
$\fa \in \AA_\fb$; this happens for only one $\fb$
and only finitely many $\fa$ by the induction hypothesis.
Or it could be that $\m \in \supp(\fa\fb^\dagger)\fb$.  This happens
for only one $\fb$ and (since both $\AA_\fb$ and
$\supp\fb^\dagger$ are well ordered) only finitely many $\fa$.
So, in all, $\m \in \supp(\g')$ for only finitely many $\g \in \AA$.

(ii) For $\fb \in \BB$, let
$$
	\FC_\fb = \big(\AA_\fb\cdot\supp\fb^\dagger\big)
	\cup \bigcup_{\fa \in \AA_\fb} \supp(\fa') .
$$
So using the induction hypothesis and
\cite[\Ewellprod]{edgar}, we conclude that
$\FC_\fb \subset \WW_{N-1}$ is well ordered.  Therefore
$$
	\bigcup_{\g \in \AA} \supp(\g')
	\subseteq \SET{\fa\fb}{\fb \in \BB, \fa \in \FC_\fb}
$$
is also well ordered since it is ordered lexicographically.

(iii)~follows from (i) and (ii).
\end{proof}

\begin{proof}[Proof of Theorem~\ref{wellderiv}]
Recall the notation $\l_m = \log\circ\log\circ\cdots\circ\log$
with $m$ logarithms ($m>0$), $\l_0 = x$,
$\l_{-m} = \exp\circ\exp\circ\cdots\circ\exp$ with $m$
exponentials.  Note for $m \ge 1$,
$\l_m' = 1/(x\l_1\l_2\cdots\l_{m-1}) \in \WW_{m-1,m-1}$.

Define
$\AA_1 = \SET{\g\circ\l_{-M}}{\g \in \AA}$. Then
$\AA_1$ is well ordered and  $\AA_1 \subseteq \WW_{N}$.
Thus the previous result applies to $\AA_1$.
Now for $\g \in \AA$
we have $\g = \g_1\circ\l_M$, $\g_1 \in \AA_1$,
and $\g' = (\g_1'\circ\l_M)\cdot\l_M'$.
So $\supp(\g') = (\supp(\g_1')\circ\l_M)\cdot\l_M'$.  Both
correspondences (compose with $\l_M$ and
multiply by $\l_M'$) are bijective and order-preserving.
So the family
$\SET{\supp(\g')}{\g \in \AA}$ is point-finite since
$\SET{\supp(\g_1')}{\g_1 \in \AA_1}$ is point-finite;
$\bigcup_{\g \in \AA} \supp(\g')$ is well-ordered since
$\bigcup_{\g_1\in \AA_1} \supp(\g_1')$ is well-ordered.
And $\supp(\g') \subset \WW_{\max(N,M),M}$,
so $T' \in \T_{\bullet,\bullet}$.
\end{proof}

Now we consider a set closed under derivative in a
certain sense: a single well ordered set that supports
all derivatives of some $T$.

\begin{pr}\label{u0}
Let $\AA \subset \WW$ satisfy:
$\AA$ is log-free; $\AA$ is well ordered; $\m^\dagger \fsteq 1$
for all $\m \in \AA$.   Then there is $\widetilde\AA$ such that:
$\widetilde\AA \supseteq \AA$; $\widetilde\AA$ is log-free;
$\widetilde\AA$ is well ordered; $\m^\dagger \fsteq 1$
for all $\m \in \widetilde\AA$;
if $\m \in \widetilde\AA$ then $\supp(\m') \in \widetilde\AA$.
\end{pr}
\begin{proof}
Let $\AA$ be log-free and well ordered
with $\m^\dagger \fsteq 1$ for all $\m \in \AA$.
Now $(e^{x^2})^\dagger = 2x \fgt 1$, so by ``height wins''
$\AA \subset \WW_1$.  We may decompose $\AA$ by factoring
each $\g \in \AA$ as $\g=x^be^L$, so that
$$
	\AA = \SET{x^be^L}{e^L \in \BB, x^b \in \AA_L},
$$
where $\BB$ is well ordered and, for each $e^L \in \BB$,
the set $\AA_L \subseteq \WW_0$ is well ordered;
the ordering is lexicographic:
$$
	x^{b_1}e^{L_1} \fst x^{b_2}e^{L_2}
	\quad\Longleftrightarrow\quad
	L_1 < L_2 \text{ or }
	\{\;L_1=L_2 \text{ and } b_1 < b_2\;\} .
$$
Now fix an $L$ with $e^L \in \BB$.  (Of course $L=0$
is allowed.) Then
$L' \fsteq 1$, so $\supp L'$ is a well ordered
set in $\WW_0$ with $\m \fsteq 1$ for all $\m \in \supp L'$.
The monoid $(\supp L')^*$ generated by $\supp L'$
is well-ordered.  So
$$
	\widetilde\AA_L :=
	(\supp L')^*\cdot\AA_L\cdot\{1,x^{-1},x^{-2},x^{-3},\cdots\}
$$
is well ordered. Define
$$
	\widetilde\AA := \SET{x^b e^L}{e^L \in \BB,
	x^b \in \widetilde{\AA}_L} .
$$
Because the ordering is lexicographic, $\widetilde\AA$
is also well ordered.
Note $\AA \subseteq \widetilde\AA \subset \WW_1$.  If
$x^b e^L \in \widetilde\AA$, then
$(x^b e^L)^\dagger \fsteq x^{-1}+L' \fsteq 1$.
Let $\m = x^b e^L \in \widetilde\AA$.  Then
$\m' = (b x^{b-1} +x^b L') e^L$.  But
$x^{b-1} \in \widetilde\AA_L$ and
$\supp(x^b L') \subseteq \widetilde\AA_L$.  Therefore
$\supp(\m') \subseteq \widetilde\AA$.
\end{proof}

Note: Let $\widetilde\AA \subseteq \WW$ with
$e^{x^2} \in \widetilde\AA$ and if $\m \in \widetilde\AA$
then $\supp(\m') \subseteq \widetilde\AA$.  Such
$\widetilde\AA$ cannot be well ordered, since
it contains $x^je^{x^2}$ for all $j \in \N$.
But there are at least the following two propositions.

\begin{pr}\label{u1}
Let $\e \in \WW_{N}\setminus \WW_{N-1}, \e \fst 1$.
Let $\AA \subset \WW_{N}$ be well ordered such that
$\m^\dagger \fsteq 1/(x\e)$ for all $\m \in \AA$.
Then there exists well ordered $\widetilde\AA \subset \WW_{N}$
such that $\widetilde\AA \supseteq \AA$ and
if $\g \in \widetilde\AA$, then
$\supp(x\e\g') \subseteq \widetilde\AA$.
\end{pr}
\begin{proof}
Write $\e = \e_0 \e_1$ with $\e_0 \in \WW_{N-1}$, $\e_1 \in \Wpure_N$,
$\e_1 \fst 1$.
Now for $\g = \fa\fb \in \AA$, we have
\begin{equation}\label{eq:derivative}
	x\e\g' = x\e_0\e_1(\fa'\fb + \fa\fb')
	= (x\e_0\fa' + x\e_0\fa\fb^\dagger)\cdot(\e_1\fb) .
\end{equation}
with $\e_1\fb \in \Wpure_N$ and support of
the first factor in $\WW_{N-1}$.  Applying this again:
$$
	(x\e\partial)^2\,\g = 
	\big[x\e_0(x\e_0\fa' + x\e_0\fa\fb^\dagger)'+
	x\e_0(x\e_0\fa' + x\e_0\fa\fb^\dagger)\fb^\dagger\big]\e_1\fb .
$$
Continue many times:  $(x\e\partial)^j \g = V\cdot \e_1^j\fb$,
$\supp V \subset \WW_{N-1}$, every term in $V$ has the following
form:  some $\fa \in \AA_\fb$, or some derivative, up to order $j$,
multiplied by factors chosen from $x,\e_1,\fb^\dagger,\e_1^\dagger$,
or derivatives of these, up to order $j$, each to a power at most $j$.
So there are finitely many well ordered sets involved.

Now let $\widetilde\BB = \BB\cdot\{1, \e_1, \e_1^2, \cdots\}$.
Thus $\BB \subseteq \widetilde\BB \subset \Wpure_N$
and $\widetilde\BB$ is well ordered.  Fix
$\m \in \widetilde\BB$.  Because $\BB$ is well-ordered
and $\e_1 \fst 1$, we have
$\m = \fb\e_1^j$ with $\fb \in \BB$ for only finitely many
different values of $j$.  For each such $j$ we get a well ordered
set in $\WW_{N-1}$.  Since there are finitely many $j$, in all
we get a well ordered set, call it $\widetilde\AA_\m$.
Our final result is
$$
	\widetilde\AA = \SET{\fa\m}{\m \in \widetilde\BB,
	\fa \in \widetilde\AA_\m},
$$
again with lexicographic order.  So $\widetilde\AA$ is well ordered.
From (\ref{eq:derivative}) we conclude:
if $\g \in \widetilde\AA$, then $\supp(x\e\g') \subseteq \widetilde\AA$.
\end{proof}

\begin{pr}\label{u2}
Let $\e \in \WW_{N}\setminus \WW_{N-1}, \e \fst 1$.
Let $\AA \subset \WW_{\bullet}$ be well ordered such that
$\m^\dagger \fsteq 1/(x\e)$ for all $\m \in \AA$.
Then there exists well ordered $\widetilde\AA \subset \WW_{\bullet}$
such that $\widetilde\AA \supseteq \AA$ and
if $\g \in \widetilde\AA$, then
$\supp(x\e\g') \subseteq \widetilde\AA$.
\end{pr}
\begin{proof}
Let $n$ be minimum such that $\AA \subset \WW_n$.
If $n=N$, then this has been proved in Proposition~\ref{u2}.  In fact, if
$n < N$ the proof in Proposition~\ref{u2} still works with $\BB = \{1\}$.
We proceed by induction on $n$.  Assume $n>N$ and
the result is true for smaller $n$.  Decompose $\AA$
as usual:
$$
	\AA = \SET{\fa \fb}{\fb \in \BB, \fa \in \AA_\fb},
$$
where $\BB \subset \Wpure_n$ is well ordered,
and for each $\fb \in \BB$, the set $\AA_\fb \subset \WW_{n-1}$
is well ordered.  For $\g = \fa\fb \in \AA$,
\begin{equation}\label{eq:induc}
	x\e\g' = (x\e\fa'+x\e\fa\fb^\dagger)\fb .
\end{equation}
Now $\supp(x\e\fb^\dagger)$ is well ordered and $\fsteq 1$, so the
monoid $\big(\supp(x\e\fb^\dagger)\big)^*$ generated by
it is well ordered, so
$\AA_\fb\cdot\big(\supp(x\e\fb^\dagger)\big)^*$ is
well ordered.  By the induction hypothesis, there
exists well ordered $\widetilde\AA_\fb$ such that
$$
	\AA_\fb\cdot\big(\supp(x\e\fb^\dagger)\big)^*
	\subseteq \widetilde\AA_\fb \subset \WW_{n-1},
$$
and if $\m \in \widetilde\AA_\fb$
then $\supp(x\e\m')\subseteq\widetilde\AA_\fb$.
Then define
$$
	\widetilde\AA = \SET{\fa\fb}{\fb \in \BB,
	\fa \in \widetilde\AA_\fb},
$$
which is again well ordered.
From (\ref{eq:induc}) we conclude:
if $\g \in \widetilde\AA$, then $\supp(x\e\g') \subseteq \widetilde\AA$.
\end{proof}

\section{The Recursive Structure of the Transline}\label{RecStr}

\begin{pr}[{Inductive Principle}]\label{inductiveprinciple}
Let  $\RR \subseteq \T$.  Assume:
\begin{itemize}
\item[{\rm(a)}] $a \in \RR$ for all constants $a \in \R$.
\item[{\rm(b)}] $x \in \RR$.
\item[{\rm(c)}] If $A, B \in \RR$, then $A B \in \RR$.
\item[{\rm(d)}] If $A_i \in \RR$ for all $i$ in some index set, and
$A_i \to 0$, then $\sum A_i \in \RR$.
\item[{\rm(e)}] If $A \in \RR$, then $e^A \in \RR$.
\item[{\rm(f)}] If $A \in \RR$, then $A \circ \log \in \RR$.
\end{itemize}
Then $\RR = \T$.
\end{pr}
\begin{proof}
This principle is clear from the definition for $\T$ in \cite{edgar}
once we observe:\hfill\break
(i)~$x \circ \log = \log(x)$, so $\log(x) \in \RR$ by (b) and (f).
(ii)~If $b \in \R$, then $b \log(x) \in \RR$ by (a) and (c).
(iii)~$e^{b\log(x)} = x^b$, so $x^b \in \RR$ by (e).
(iv)~Once the terms of a purely large $L$ are known to be in $\RR$,
we get monomial $x^b e^L \in \RR$.
(v)~If $T = \sum c_j \g_j$ and monomials $g_j \in \RR$,
then $T \in \RR$.
(vi)~If $T \in \RR$, then $T \circ \log_M \in \RR$.
\end{proof}

In fact, the set of conditions can be reduced:
\begin{co}\label{inductivecor}
Let $\RR \subseteq \T = \R\lbb \G \rbb$, and identify $\G$
as a subset of $\T$ as usual.  Assume:
\begin{itemize}
\item[{\rm(d$'$)}] If $\supp A \subseteq \RR$, then $A \in \RR$.
\item[{\rm(e$'$)}] If $b \in \R$ and $L \in \RR$ is purely large and
log-free, then $x^b e^L \in \RR$.
\item[{\rm(f$\,'$)}] If $\g \in \RR$ is a monomial, then
$\g \circ \log \in \RR$.
\end{itemize}
Then $\RR = \T$.
\end{co}
\begin{proof}
Since $\supp 0 = \emptyset$,
we get $0 \in \RR$ by (d$'$); but $0$ is purely large
and log-free, so $1, x \in \RR$ by (e$'$).  Follow the
construction in \cite{edgar}.
\end{proof}

Another inductive form (see \cite{hoeven}):

\begin{co}\label{inductivelog}
Let $\RR \subseteq \T = \R\lbb \G \rbb$, and identify $\G$
as a subset of $\T$ as usual.  Assume:
\begin{itemize}
\item[{\rm(b$''$)}] For all $n \in \N$, $\l_n \in \RR$.
\item[{\rm(d$''$)}] If $\supp A \subseteq \RR$, then $A \in \RR$.
\item[{\rm(e$''$)}] If $L \in \RR$ is purely large, then $e^L \in \RR$.
\end{itemize}
Then $\RR = \T$.
\end{co}
\begin{proof}
First, $\log x \in \RR$ by (b$''$).  For any $b \in \R$,
$b \log x$ is purely large, so $e^{b \log x} = x^b \in \RR$
by (e$''$).  Next,
$\T_0 \subseteq \RR$ and
$b \log x + L \in \RR$ for any purely large $L \in \T_0$
by (d$''$), so $e^{b\log x + L} = x^b e^L \in \RR$.
Thus $\G_1 \subseteq \RR$ so $\T_1 \subseteq \RR$.  Continuing inductively,
$\G_n, \T_n \subseteq \RR$ for all $n \in \N$.
So $\T_\bullet \subseteq \RR$.

Note that $\widetilde\RR := \SET{T \in \T}{T \circ \log \in \RR}$
also satisfies the three conditions, so by the preceding
paragraph $\T_\bullet \subseteq \widetilde\RR$,
and $\T_{\bullet1} \subseteq \RR$.  Continuing inductively,
$\T_{\bullet m} \subseteq \RR$ for all $m \in \N$.
So $\T_{\bullet\bullet} \subseteq \RR$ and $\RR = \T$.
\end{proof}

\begin{qu}
Is there a good recursive formulation for $\LP$ or $\SS$?
See~\ref{no:LPSS}.
\end{qu}

\subsection*{The Schmeling Tree of a Transmonomial}
Let $\g$ be a transmonomial, $\g\in\G$.  Then
$\g = e^L$, where $L \in \T$ is purely large.
So $L = c_0 \g_0 + c_1\g_1 + \cdots$ where $c_i \in \R$
and $\g_i \in \Glarge$.  We may index this as
$L = \sum_i c_i \g_i$, where $i$ runs over some ordinal
(an ordinal $<\omega^\omega$ for the grid-based case;
just countable for the well-based case; possibly finite;
possibly just a single term; or even
no terms at all if $\g=1$).

In turn,
each $\g_i = e^{L_i}$, where $L_i \in \T$ is purely large
and positive.  So
$L_i = \sum_j c_{ij} \g_{ij}$, where index $j$ runs over
some ordinal (possibly a different ordinal for different $i$).
Continuing, each $\g_{ij} = e^{L_{ij}}$, where
$L_{ij} \in \T$, and $L_{ij} = \sum_k c_{ijk}\g_{ijk}$
where $\g_{ijk} \in \G$.
And so on:
each $\g_{i_1i_2\dots i_s}$ is in $\Glarge$, and has
the form $\g_{i_1 i_2\cdots i_s} = e^{L_{i_1 i_2\cdots i_s}}$,
and $L_{i_1 i_2\cdots i_s} =
\sum_{j} c_{i_1 i_2\cdots i_s j} \g_{i_1 i_2\cdots i_s j}$.

Say the original monomial $\g$ has height $N$; that is,
in the terminology of \cite{edgar},
$\g \in \G_{N,\bullet}$.  Then
eventually (with $s \le N$) we reach
$\g_{i_1 i_2\cdots i_s} = (\l_m)^b$ for some $m$, and
if $b \ne 1$, then in one more step we get
$\g_{i_1 i_2\cdots i_{s+1}} = \l_{m+1}$.
Let us stop a ``branch'' $i_1, i_2, \cdots$ when we
reach some $\l_m$ (even if $m \le 0$ so that we
have $x$ or $\exp_n x$).

The structure of the monomial $\g$ then corresponds to a
\emph{Schmeling tree}.  (We have adapted this tree discription
from Schmeling's thesis \cite{schmeling}.)
Each node corresponds to some monomial.
The root corresponds to $\g$.  The children of $\g$ are
the $\g_i$. A leaf corresponds to some
$\log_m x$, and is labeled by the integer $m$.  Each node
that is not a leaf has countably many children, arranged
in an ordinal, and each edge is labeled by a real number.
All nodes $\g_{i_1 i_2\cdots i_s}$ in the tree
(except possibly the root $\g$) are large monomials.

\begin{ex}
Consider the following example.
The ordinals here are all finite, so that
everything can be written down.
\begin{align*}
	\g &= e^{\displaystyle
	-e^{\displaystyle 4e^{\displaystyle 2 x^4-x}
	-(2/3)e^{\displaystyle x}}+3e^{\displaystyle \pi
	e^{\displaystyle x^4-2x^2}+\log x}
	}
\\ &=\exp\Bigg(
	-\exp\bigg(4\exp\Big(2 x^4-x\Big)
	-(2/3)\exp x\bigg)
	\\ &\qquad\qquad\qquad +3\exp\bigg(\pi
	\exp\Big( x^4-2x^2\Big)+\log x\bigg)
	\Bigg) .
\end{align*}
The component parts of the tree:
\begin{align*}
	\g_0 &= 
	e^{\displaystyle 4e^{\displaystyle 2 x^4-x}
	-(2/3)e^{\displaystyle x}},
	c_0 = -1,\quad
	\g_1 = e^{\displaystyle \pi
	e^{\displaystyle x^4-2x^2}+\log x},
	c_1 = 3,
\\
	\g_{00} &= e^{\displaystyle 2 x^4-x},
	c_{00} = 4,\quad
	\g_{01} = e^x = \log_{-1} x,
	c_{01} = -2/3,
\\
	\g_{10} &= e^{\displaystyle x^4-2x^2}, c_{10} = \pi,\quad
	\g_{11} = \log x = \log_1 x, c_{11} = 1,
\\
	\g_{000} &= x^4 = e^{4\log x}, c_{000} = 2,\quad
	\g_{001} = x = \log_0 x, c_{001} = -1 ,
\\
	\g_{100} &= x^4 = e^{4\log x}, c_{100} = 1,\quad
	\g_{101} = x^2 = e^{2\log x}, c_{101} = -2,
\\
	\g_{0000} &= \g_{1000} = \g_{1010} = \log x = \log_1 x,
	 c_{0000} = c_{1000} = 4, c_{1010} = 2 .
\end{align*}
The tree representing $\g$ is shown in Figure~\ref{fig:graph}.
\end{ex}

\begin{figure*}[htbp] 
   \centering
   \includegraphics{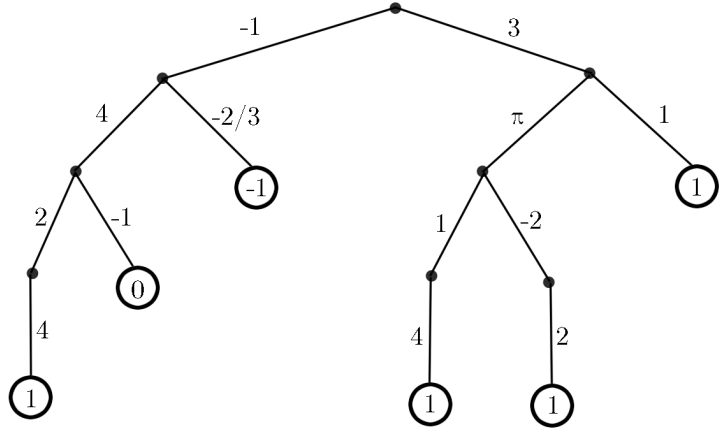} 
   \caption{The Schmeling tree corresponding to monomial $\g$}
   \label{fig:graph}
\end{figure*}

There are notions of ``height'' and ``depth'' associated with
such a tree-representation of a transmonomial $\g$.
Let us say that $\g$ has \Def{tree-height} $N$ iff the
longest branch (from root to leaf) has $N$ edges;
and that $\g$ has \Def{tree-depth} $M$ iff $M$ is the
largest label on a leaf.  So the example in
Figure~\ref{fig:graph} has tree-height $4$ and
tree-depth $1$.  These definitions are convenient for
analysis of such a tree diagram.  They may differ from
the notions of
``height'' and ``depth'' defined in \cite{edgar}.
If $\g$ has height $N$ (that is, $\g \in \G_{N\bullet}$),
then $\g$ has tree-height at most $N+1$.  But it may
be much smaller; for example,
$$
	\g = e^{\displaystyle e^{e^x}+x}
$$
has tree-height $1$ but height $3$.
If $\g$ has depth $M$ (that is, $\g \in \G_{\bullet M}$),
then $\g$ has tree-depth $M$ or $M+1$,
at least if we have allowed negative values of $M$.
The same example $\g$ has depth $0$ and tree-depth $0$,
but
$$
	\g = e^{\displaystyle e^{e^x}+x^2}
$$
has depth $0$ and tree-depth $1$.

Tree-height and tree-depth behave in the same way as
height and depth under composition on the right by $\log$
or $\exp$.  That is: if $\g$ has tree-height $N$ and
tree-depth $M$, then
$\g\circ\exp$ has tree-height $N$ and tree-depth $M-1$,
and $\g\circ\log$ has tree-height $N$ and tree-depth $M+1$.
Any $\g \in \G_0$ has tree-depth $\le -1$, so
$\g\circ \exp$ has tree-depth $\le 0$.  If tree-depth
is $\le 0$ is it sometimes convenient to extend all branches
(using single edges with coefficient $1$) so that all
leaves are $x$.

\subsection*{Schmeling Tree and Deriviative}
Let $\g$ be a transmonomial represented as a Schmeling tree.  What are the
monomials in the support of the derivative $\g'$?  Since
$\g = e^{L}$, the derivative is $e^L L'$, so the monomials
in its support have the form $\g$ times a monomial in
the support of $L'$.  Continuing this recursively,
we see that a monomial in $\supp \g'$ looks like
\begin{equation*}
	\g \;\g_{i_1}\; \g_{i_1 i_2} \;
	\cdots\;\g_{i_1 i_2 \cdots i_{s-1}}\; (\log_m x)'
\tag{1}
\end{equation*}
where $s$ is chosen so that $\g_{i_1 i_2 \cdots i_s} = \log_m x$,
and of course $(\log_m x)'$ is itself a monomial.
(The monomials
$\g_{i_1}, \cdots, \g_{i_1 i_2 \cdots i_{s-1}}$ are large, but
if $m > 0$, then the monomial $(\log_m x)'$ is small.)
So there is
one term of $\g'$ for each branch (from root
to leaf) of the tree.
In the derivative $\g'$, the coefficient for monomial (1) is
$$
	c_{i_1}\; c_{i_1 i_2} \; \cdots\;c_{i_1 i_2 \cdots i_{s}} ,
$$
the product of all the edge-labels on the corresponding branch.

\begin{ex}
Following the tree in the example (Figure~\ref{fig:graph}),
we may write
the derivative $\g'$ with one term for each of the six branches
of the tree:
\begin{align*}
	\g' =\quad &\;
	(-1)\cdot4\cdot2\cdot4\cdot\g\g_0\g_{00}\g_{000}\cdot (\log x)'
	\\ +&\;
	(-1)\cdot4\cdot(-1)\cdot\g\g_0\g_{00}\cdot x'
	\\ +&\;
	(-1)\cdot(-2/3)\cdot\g\g_0\cdot(\exp x)'
	\\ +&\;
	3\cdot\pi\cdot1\cdot4\cdot\g\g_1\g_{10}\g_{100}\;\cdot(\log x)'
	\\ +&\;
	3\cdot\pi\cdot(-2)\cdot2\cdot\g\g_1\g_{10}\g_{101}\;\cdot(\log x)'
	\\ +&\;
	3\cdot1\cdot\g\g_1\cdot(\log x)' .
\end{align*}
\end{ex}

The monomial (1) without the first factor $\g$ is an
element of the set $\lsupp(\g)$.  The magnitude of $\g'$
is the monomial we get following the left-most branch
\begin{equation*}
	\g \;\g_{0}\; \g_{00} \;
	\cdots\;\g_{00 \cdots 0}\; (\log_m x)' ,
\end{equation*}
since all other branches are far smaller.
	
In the special case where the tree-depth of $\g$ is $\le 0$, and
we extend all branches so that all leaves are $x$, the monomials
in $\g'$ are
\begin{equation*}
	\g \;\g_{i_1}\; \g_{i_1 i_2} \;
	\cdots\;\g_{i_1 i_2 \cdots i_{s-1}}
\tag{0}
\end{equation*}
where $s$ is chosen with $\g_{i_1 i_2 \cdots i_s} = x$.
In this case, all monomials
$\g_{i_1} \cdots \g_{i_1 i_2 \cdots i_{s-1}}$ in $\lsupp \g$
are large, and we have
$$
	\m :=
	\max \lsupp \g = \g_{0}\g_{00}\cdots\g_{00 \cdots 0}
	= \mag(\g'/\g) .
$$
Then $\g' \sim \g\m$, and we get
$\g^{(n)} \sim \g \m^n$ for all $n \in \N$
by induction using $\m^2 \fgt \m'$ \cite[\EWKB(iv)]{edgar}.
(This may not hold when $\g$ has positive tree-depth.)

\begin{pr}  Let $T,V \in \T$.  Assume all monomials in $T$
have tree-depth $\le 0$, and $V \fst 1/\m$ where
$\m = \max\lsupp T$.  Then
$$
	T^{(n)} V^n,\qquad n \in \N
$$
is point-finite, so the series
$$
	\sum_{n=0}^\infty T^{(n)}(x) \frac{V^n}{n!}
$$
converges in the asymptotic topology.
\end{pr}
\begin{proof}
Fix finite set $\bmu \subset \Gsmall$ so that
all far-smaller inequalities are witnessed by $\bmu$:
in particular, $V \fst^{\ebmu} 1/\m$ and
$T = \dom(T)\cdot(1+S)$ with $S \fst^{\ebmu} 1$.  Note
that $T^{(n+1)} \sim \m T^{(n)}$.  Then
$$
	T \fgt^{\ebmu} T'V \fgt^{\ebmu} T''V^2 \fgt^{\ebmu} \dots,
$$
so by \cite[\Edominateprop]{edgar}
the series $\sum T^{(n)}(x) V^n/n!$
is point-finite.
\end{proof}

\begin{re}
The same result should be true for
other $T$, perhaps using $\tsupp$ not
$\lsupp$; see \cite[\Wtsupp]{edgarw}.
\end{re}

\section{Properties of Composition}
Composition $T \circ S$ is defined when $T,S \in \T$ and $S$
is large and positive.  As usual we will
write $T = T(x)$ and $T\circ S = T(S)$.

\begin{no}\label{no:LPSS}
Write $\LP$ for the group of large positive transseries.
And $\SS$ for the subgroup $\SS =
x + \o(x) = \SET{T \in \T}{\dom T = x}
= \SET{T \in \T}{T \sim x}$.
For now, think of $\LP$ and $\SS$ as sets.
They are closed under composition.
For existence of inverses: well-based,
Proposition~\ref{inverse}; grid-based,
\cite[\Winv]{edgarw}.
\end{no}

Many basic properties of composition may be proved by
applying an inductive principle such as Proposition~\ref{inductiveprinciple}
to the left composand $T$. (I may---perhaps misleadingly---call
this ``induction on the height''.)  Here are some examples.

\begin{pr}\label{ineq1}
Let $T, T_1, T_2 \in \T, S \in \P$.  Then
{
\allowdisplaybreaks
\begin{align*}
   T >0 \Longrightarrow{}& T \circ S >0,
\\ T =0 \Longrightarrow{}& T \circ S =0,
\\ T <0 \Longrightarrow{}& T \circ S <0,
\\ T_1 < T_2 \Longrightarrow{}& T_1 \circ S < T_2 \circ S,
\\ T_1 = T_2 \Longrightarrow{}& T_1 \circ S = T_2 \circ S,
\\ T_1 > T_2 \Longrightarrow{}& T_1 \circ S > T_2 \circ S,
\\ T \fst 1 \Longrightarrow{}& T \circ S \fst 1,
\\ T \fgt 1 \Longrightarrow{}& T \circ S \fgt 1,
\\ T \fe 1 \Longrightarrow{}& T \circ S \fe 1,
\\ T \sim 1 \Longrightarrow{}& T \circ S \sim 1,
\\ T_1 \fst T_2 \Longrightarrow{}& T_1 \circ S \fst T_2 \circ S,
\\ T_1 \fgt T_2 \Longrightarrow{}& T_1 \circ S \fgt T_2 \circ S,
\\ T_1 \fe T_2 \Longrightarrow{}& T_1 \circ S \fe T_2 \circ S,
\\ T_1 \sim T_2 \Longrightarrow{}& T_1 \circ S \sim T_2 \circ S,
\\ T\circ S \fe \mag(T \circ S) ={}& \mag((\mag T)\circ S) \fe (\mag T)\circ S,
\\ T\circ S \sim \dom(T \circ S) ={}& \dom((\dom T)\circ S)
\sim (\dom T)\circ S.
\end{align*}
}
\end{pr}

Some corresponding things may fail for the other composand:
Let $T=e^x$, $S_1=x+\log x$, and $S_2=x$.
Then $S_1 \fe S_2$ but $T\circ S_1 \fgt T\circ S_2$;
$\dom(T \circ S_1) \not\fe T\circ \dom S_1$.

\begin{pr}\label{ineq}
Let $S_1, S_2 \in \LP$, $S_1 < S_2$.
\begin{itemize}
\item[{\rm(a)}] if $c \in \R, c>0$, then $S_1^c < S_2^c$,
\item[{\rm(b)}] if $c \in \R, c<0$, then $S_1^c > S_2^c$,
\item[{\rm(c)}] $\log(S_1) < \log(S_2)$.
\item[{\rm(d)}] $\exp(S_1) < \exp(S_2)$,
\end{itemize}
\end{pr}
\begin{proof} (a)
Write the canonical multiplicative decomposition
$S_1 = a_1 e^{L_1} (1+U_1)$ as in \ref{C_mult},
and similarly $S_2 = a_2 e^{L_2} (1+U_2)$.
Then
\begin{equation*}
	S_1^c = a_1^c e^{cL_1} \left(1+cU_1+\sum_{j=2}^\infty c_jU_1^j\right) ,
	\quad
	S_2^c = a_2^c e^{cL_2} \left(1+cU_2+\sum_{j=2}^\infty c_jU_2^j\right) ,
\tag{1}
\end{equation*}
for certain (binomial) coefficients $c_j$.
Now for
$S_1 < S_2$ there are these cases:
(i)~$L_1<L_2$;
(ii)~$L_1=L_2, a_1<a_2$;
(iii)~$L_1=L_2, a_1=a_2, U_1<U_2$.
But in each of these cases, applying equations (1)
shows $S_1^c < S_2^c$.  For case~(iii):
$$
	S_2^c-S_1^c = a_1^c e^{cL_1} (U_2-U_1)\left(
	c+\sum_{j=2}^\infty c_j(U_2^{j-1}+U_2^{j-2}U_1+
	\cdots+U_1^{j-1})\right) > 0
$$
since the terms in the $\sum$ are all ${}\fst 1$.

(b) is similar.

(c) Write canonical multiplicative decomposition
$S_1 = a_1 e^{L_1} (1+U_1)$ as in \ref{C_mult} and
similarly $S_2 = a_2 e^{L_2} (1+U_2)$.  Then
\begin{align*}
	\log(S_1) &= \log(a_1) + L_1 + U_1 + \sum_{j=2}^\infty c_j U_1^j ,
	\\
	\log(S_2) &= \log(a_2) + L_2 + U_2 + \sum_{j=2}^\infty c_j U_2^j , 
\end{align*}
for certain coefficients $c_j$.
The same cases (i)---(iii) may be used, and in each case
we get $\log(S_1) < \log(S_2)$.  Case (iii) has reasoning as
we did before for~(a).

(d) For this, write the canonical additive decomposition
$S_1 = L_1 + c_1 + U_1$ as in \ref{C_add},
and similarly $S_2 = L_2 + c_2 + U_2$.  Then
\begin{equation*}
	e^{S_1} = e^{c_1} e^{L_1} \left(1+U_1+\sum_{j=2}^\infty c_j U_1^j\right),
	\qquad
	e^{S_2} = e^{c_2} e^{L_2} \left(1+U_2+\sum_{j=2}^\infty c_j U_2^j\right),
\end{equation*}
for certain coefficients $c_j$.
For $S_1<S_2$ there are three cases:
(i)~$L_1<L_2$;
(ii)~$L_1=L_2, c_1<c_2$;
(iii)~$L_1=L_2, c_1=c_2, U_1<U_2$.
In all three cases we get $e^{S_1} < e^{S_2}$.
\end{proof}

\begin{pr}\label{explog2}
{\rm(a)}~If $T \in \T$, $T>0$, $T\ne 1$, then $\log T < T - 1$.
{\rm(b)}~If $T \in \T$, $T \ne 0$, then $\exp T > T + 1$
\end{pr}
\begin{proof}
First note: If $L$ is purely large and positive, then
$e^L \fgt L$.  First use \cite[\Eheightwins]{edgar} for
log-free $L$.  Then Proposition~\ref{ineq1} to compose
with $\log_M$ on the inside.  It follows that:
If $T \fgt 1$ and $T>0$, then $e^T \fgt T$.

(a) Write $A = \log(T) - T + 1$; I must show $A < 0$.
Write canonical multiplicative decomposition
$T = a e^L(1+U)$ as in \ref{C_mult}.  Then
$\log(T) = \log(a) + L - \sum_{j=1}^\infty (-1)^jU^j/j$.  Now if
$L>0$, then $T \fgt 1$, $T \fgt L \fgt 1$, so
$A \sim -T < 0$.  If $L<0$, then $T \fst 1 \fst L$,
so $A \sim L < 0$.  So assume $L=0$.  Now if
$a \ne 1$, then $A \sim \log(a)-a+1$, which is $<0$
by the ordinary real Taylor theorem.  So assume $a=1$.
Then if $U \ne 0$ we have
$$
	A = - \sum_{j=1}^\infty \frac{(-1)^j U^j}{j} -(1+U)+1
	= -\frac{U^2}{2} + \o(U^2) < 0.
$$
So the only case left is $U=0$, and that means $T=1$.

(b) Write $A = \exp T - T - 1$; I must show $A > 0$.
Write canonical additive decomposition
$T = L+c+V$ as in \ref{C_add}.  So $\exp T = e^L e^c (1+V+\dots)$.
If $L > 0$, then $T \fgt 1$, $e^T \fgt T \fgt 1$,
so $A \sim e^T > 0$.  If $L<0$, then
$e^T \fst 1$, $T \sim L \fgt 1$, so $A \sim -L < 1$.
So assume $L=0$.  If $c \ne 0$, then
$A \sim e^c-c-1$, which is $>0$ by the ordinary real Taylor theorem.
So assume $c=0$.  Then if $V \ne 0$ we have
$$
	A = \sum_{j=0}^\infty \frac{V^j}{j!} - V - 1
	= \frac{V^2}{2} + \o(V^2) > 0 .
$$
So the only case left is $V=0$, and that means $T = 0$.
\end{proof}

\subsection*{Exponentiality}
Associated to each large positive transseries is an integer
known as its ``exponentiality'' \cite[Exercise~4.10]{hoeven}.
If you compose with $\log$ sufficiently many times on the left, the
magnitude is a leaf $\l_m$.  The number $p$ in the following
result is the \Def{exponentiality} of $Q$, written
$p = \expo Q$.

\begin{pr}\label{exponentiality}
Let $Q \in \LP$.  Then there is $p \in \Z$ and $N \in \N$
so that for all $n \ge N$,
$\log_n \circ\;Q \circ \exp_n \sim \exp_p$.  Equivalently,
$\log_n Q \sim \l_{n-p}$.
\end{pr}
\begin{proof}
We will use the basic definition for logarithms.
Let $A = c e^L(1+U)$ be the canonical multiplicative
decomposition.  If $A \in \P$, this means $c>0$ and
$L$ is purely large and positive.  Then
$\log A = L + \log c + \sum_{j=1}^\infty ((-1)^{j+1}/j)U^j$.
From this we get: If $A,B \in \LP$, $A \fe B$, then
$\log A \sim \log B$.  Write
$\RR[p,N] := \SET{Q \in \LP}{\log_n Q \sim \l_{n-p}
\text{ for all } n \ge N}$.

(i) $\l_m \in \RR[-m,0]$.

(ii) Let $A = c e^L(1+U) \in \LP$, then
$\dom(\log A) = \dom L$, where also $\dom L \in \LP$ and (unless
$L$ has height $0$) the height of $\dom L$ is less
than the height of $\dom A = ce^L$.  If $\dom L \in \RR[p,N]$ then
$A \in \RR[p+1,N+1]$.

(iii) Let $A$ have height $0$, so $A \sim c \l_m^b$, $c,b \in \R$,
$c>0$, $b>0$.  Then $\log A \sim b \l_{m+1}$ and
$\log_2 A \sim \l_{m+2}$, so $A \in \R[-m,2]$.

These rules cover all $\LP$.
\end{proof}

\begin{re}
Alternate terminology: exponentiality = level.
So Proposition~\ref{exponentiality} says that
the exponential ordered field $\R\lbbb x \rbbb$
is \Def{levelled}.
\end{re}

\begin{ex}
$$
	T \sim 4(\log x)^2 x^{\pi} e^{5 x^2-x}
$$
(so that the dominant term of $T$ is
$4(\log x)^2 x^{\pi} e^{5 x^2-x}$), then
\begin{align*}
	\log \circ\;T \circ \exp &\sim
	5 e^{2x} - e^x + \pi x + 2 \log x + \log 4 \sim 5 e^{2x} ,
	\\
	\log_2 \circ\;T \circ \exp_2 &\sim
	2 e^x + \log 5 \sim 2 e^x ,
	\\
	\log_3 \circ\;T \circ \exp_3 &\sim e^x + \log 2 \sim e^x ,
	\\
	\log_k \circ\;T \circ \exp_k &\sim e^x ,\qquad\qquad\text{for all } k \ge 3 ,
\end{align*}
so $\expo T = 1$.
\end{ex}

\begin{pr}\label{klargelogfree}
If $\expo T=0$, then $\log_{k} \circ\;T \circ \exp_{k}$
is log-free for $k$ large enough.
\end{pr}
\begin{proof}
Prove recursively:
Assume $T = x+A$, $A \in \R\lbb\G_{\bullet,M}\rbb$, $M > 0$, $A \fst x$.
Then $T \circ \exp = e^x + A\circ \exp = e^x(1+B)$
with $B = (A/x)\circ \exp \in \R\lbb\G_{\bullet,M-1}\rbb$ and
$\log \circ\; T \circ \exp = x + \sum_{j=1}^\infty (-1)^{j+1}B^j/j$
has depth $M-1$.
\end{proof}

\subsection*{Simpler Proof Needed}
Here is a simple fact.  It needs a simple proof.  It is true
for functions, so it is surely true for transseries
as well.  My overly-involved proof will be given
in Section~\ref{involvedproof}.
In fact, there are two propositions.
Each can be deduced from the other:

\begin{pr}\label{posderiv}
Let $T \in \T$, $S_1,S_2 \in \P$, $S_1 < S_2$.  Then
\begin{align*}
   T' > 0 \Longrightarrow{}& T\circ S_1 < T \circ S_2,
\\ T' = 0 \Longrightarrow{}& T\circ S_1 = T \circ S_2,
\tag{1}
\\ T' < 0 \Longrightarrow{}& T\circ S_1 > T\circ S_2 .
\end{align*}
\end{pr}

\begin{pr}\label{derivcompare}
Let $A, B \in \T$, $S_1,S_2 \in \P$, $A' \fst B'$,
$S_1 < S_2$.  Then
\begin{equation*}
	A \circ S_2 - A \circ S_1 \fst 
	B \circ S_2 - B \circ S_1 .
\tag{2}
\end{equation*}
\end{pr}

\begin{proof}[Proof of \ref{derivcompare} from \ref{posderiv}]
Since the theorem is unchanged
when we replace $B$ by $-B$, we may assume $B' > 0$.
We have $A' \fst B'$.  Let $c \in \R$.  By
Remark~\ref{relationship}, $B' > c A'$ so
$(B - cA)'> 0$.  Therefore, by Proposition~\ref{posderiv},
$(B - cA)\circ S_1 < (B - cA)\circ S_2$ so
$$
	B\circ S_2 - B\circ S_1 >
	c\big(A\circ S_2 - A\circ S_1\big) .
$$
This is true for all $c \in \R$, so we have
$B\circ S_2 - B\circ S_1
\fgt A\circ S_2 - A\circ S_1$.
\end{proof}
\begin{proof}[Proof of \ref{posderiv} from \ref{derivcompare}]
Let $\RR$ be the set of all $T \in \T$ that
satisfy (1) for all $S_1,S_2\in \P$ with $S_1<S_2$.
We claim $\RR$ satisfies the conditions of
Corollary~\ref{inductivelog}.
Clearly $1, x \in \RR$.

(b$''$) Note $\l_m' = 1/\prod_{j=0}^{m-1} \l_j > 0$.
If $S_1 < S_2$, then by Proposition \ref{ineq}(c)
we have $\log_m S_1 < \log_m S_2$.

(d$''$) Assume $\supp T \subseteq \RR$.  If $T=0$, the conclusion
is clear.  Assume $T \ne 0$.  Let $a\g = \dom T$, $a \in \R$,
$\g \in \G$.  We may assume $\g \ne 1$, since
if $\g=1$, we may consider $T - a\g$ instead.  So
$T' \sim a \g'$.  Write $A = T-a\g$ so that
$T = a\g + A$ with $A \fst a\g$.
There will be cases based on the signs of $a$ and $\g'$.
Take the case $a>0, \g'>0$.
So $\g\circ S_1 < \g\circ S_2$ since $\g \in \RR$.
Now by Proposition~\ref{derivcompare},
$$
	a \g\circ S_2 - a\g\circ S_1 \fgt
	A\circ S_2 - A\circ S_1 ,
$$
so $T\circ S_2 - T\circ S_1 \sim a \g\circ S_2 - a\g\circ S_1 > 0$
and therefore $T\circ S_2 - T\circ S_1 > 0$.
The other three cases are similar.

(e$''$)  Let $T =e^L$, where
$L \in \RR$ is purely large.
Then $T' = L' e^L$, so $T'$ has the same
sign as $L'$.  Thus $L \circ S_1 < L \circ S_2$
if $T'>0$ and reversed if $T'<0$.  Apply Proposition \ref{ineq}(d) to get
$e^{L\circ S_1} < e^{L\circ S_2}$ or reversed, as required.
\end{proof}

\begin{re}
To prove either \ref{derivcompare} or \ref{posderiv} outright seems
to require more work than the proofs found above.  See
Theorem~\ref{posderivthm}.
\end{re}

Here is a special case of Proposition~\ref{derivcompare}.

\begin{pr}\label{mvt1}
If $A \in \T, S_1, S_2 \in \P$, $S_1 < S_2$, and $A \fst x$, then
\hfill\break
$A \circ S_2 - A \circ S_1 \fst S_2 - S_1$.
\end{pr}
\begin{proof}
Note $A' \fst x'$
and apply Proposition~\ref{derivcompare}.
\end{proof}

\subsection*{Grid-Based Version}
As we know, $T \fst S$ if and only if $T \fst^\ebmu S$ for some
finite set $\bmu \subset \Gsmall$ of generators.  So of course
Proposition~\ref{mvt1} needs
a form in terms of ratio sets.  It is found in
\cite[\Wfuppermonoc]{edgarw}:

\begin{pr}\label{fuppermonoc}
Let $\bmu$ be a ratio set.  Let $S_1, S_2 \in \P$.
Then there is a ratio set $\ba$ such that:
For every $A \in \T^\ebmu$, if $A \fst^\ebmu x$,
then $A(S_2) - A(S_1) \fst^\ba S_2-S_1$.
\end{pr}

Note that $\ba$ depends on $S_1$ and $S_2$, not
just on a ratio set generating them.
It is apparently not possible to avoid this problem:

\begin{qu}
Given a ratio set $\bmu \subset \Gsmall$,
is there $\ba \supseteq \bmu$ such that:
if $A,S_1,S_2 \in \T^{\ebmu}$, $A \fst^{\ebmu} x$,
$S_1,S_2 \in \P$, and $S_1 < S_2$, then
$A \circ S_2 - A \circ S_1 \fst^{\ba} S_2 - S_1$?
\end{qu}

\begin{ex}
Let $\bmu = \{x^{-1},e^{-x^3}\}$.  Consider
$A = \mu_2 = e^{-x^3}$ and
$S_a = \mu_1^{-1}+a\mu_1 = x+a x^{-1}$
for $a \in \R$.  Certainly $A \fst^\ebmu 1$.
Compute
\begin{align*}
	A \circ S_a
	&= e^{-(x+a x^{-1})^3}
	=
	e^{-x^3 -3 a x - 3 a^2 x^{-1} - a^3 x^{-3}}
	\\ &=
	e^{-x^3 -3 a x }\left(\sum_{j=0}^\infty
	\frac{(- 3 a^2 x^{-1} - a^3 x^{-3})^j}{j!}\right) .
\end{align*}
The dominant term is the monomial $e^{-x^3 -3 a x }$.
As $a$ ranges over $\R$, these monomials do not lie in any
grid.  Nor even in any well ordered set.

Now if $a<b$, then $S_a < S_b$ and
$e^{-x^3 -3 a x } \fgt e^{-x^3 -3 b x }$,
so
$$
	S_b - S_a = (b-a)x^{-1},\qquad
	A \circ S_b - A \circ S_a \sim -e^{-x^3 -3 a x } .
$$
Of course $A \circ S_b - A \circ S_a \fst S_b - S_a$.
But there is no finite $\ba$ such that
\hfill\break
$A \circ S_b - A \circ S_a \fst^{\ba} S_b - S_a $
for all $a,b$ ranging over the reals.
\end{ex}

\subsection*{Integral Notation}
\begin{no}
If $A, B \in \T$ and $A' = B$, we may
sometimes write $A = \int B$, but in fact $A$
is only determined by $B$ up to a constant summand.
The large part of $A$ is determined by $B$.
We also write $\int_{S_1}^{S_2} B := A(S_2)-A(S_1)$,
which is uniquely determined by $B$, and is defined for
$S_1, S_2 \in \P, S_1 < S_2$.
\end{no}

Of course, with this definition, any statement about integrals
is equivalent to a statement about derivatives.
Propositions~\ref{posderiv} or \ref{derivcompare}
lead to the following.

\begin{co}
Let $A, B \in \T$,
$S_1,S_2 \in \P$, $S_1 < S_2$.  Then
{
\allowdisplaybreaks
\begin{align*}
   B>0 \Longrightarrow{}& \int_{S_1}^{S_2} B > 0,
\\ B=0 \Longrightarrow{}& \int_{S_1}^{S_2} B = 0,
\\ B<0 \Longrightarrow{}& \int_{S_1}^{S_2} B < 0 .
\\ A>B \Longrightarrow{}&
\int_{S_1}^{S_2} A > \int_{S_1}^{S_2} B,
\\ A=B \Longrightarrow{}&
\int_{S_1}^{S_2} A = \int_{S_1}^{S_2} B,
\\ A<B \Longrightarrow{}&
\int_{S_1}^{S_2} A < \int_{S_1}^{S_2} B.
\end{align*}
}
\end{co}

Remark~\ref{relationship} lets us
prove formulas about $\fst$ from formulas about $<$.
Here are some examples.

\begin{pr}
If $A,B \in \T$, $A,B$ nonzero,
$S_1,S_2 \in \P$, $S_1 < S_2$, then
\begin{align*}
A \fgt B \Longrightarrow{}&
\int_{S_1}^{S_2} A \fgt \int_{S_1}^{S_2} B,
\\ A \fst B \Longrightarrow{}&
\int_{S_1}^{S_2} A \fst \int_{S_1}^{S_2} B,
\\ A\fe B \Longrightarrow{}&
\int_{S_1}^{S_2} A \fe \int_{S_1}^{S_2} B,
\\ A\sim B \Longrightarrow{}&
\int_{S_1}^{S_2} A \sim \int_{S_1}^{S_2} B.
\end{align*}
\end{pr}

\subsection*{Compositional Inverse}
Now using Proposition~\ref{mvt1}
we get a nice proof for the existence of inverses under
composition.  (For the well-based case.)
See also \cite[Cor.~6.25]{DMM}.

\begin{pr}\label{inverse1}
Let $T = x+A$, $A \fst x$, $\supp A \subset \G_N$.  Then $T$
has an inverse $S$ under composition, $S = x+B$, $B \fst x$,
$\supp B \subset \G_N$.
\end{pr}
\begin{proof}
Let the function $\Phi$ be defined by $\Phi(S)=x-A\circ S$.
Then $\Phi$ maps $\SA := \SET{x+B}{B \fst x, \supp B \subseteq \G_N}$
into itself \cite[\EcomponN]{edgar}.
I claim $\Phi$ is contracting on $\SA$.
Indeed, if $S_1, S_2 \in \SS$ and $S_1 \ne S_2$, then
$$
	\Phi(S_2) - \Phi(S_1) = A\circ S_1 - A\circ S_2 \fst S_2-S_1
$$
by Proposition~\ref{mvt1}.

Apply the fixed-point theorem \cite[Thm.~4.7]{hoevenop}
(see Proposition~\ref{wellfixed}, below)
to get $S$ with $S = \Phi(S)$.  Then
$$
	T \circ S = S + A\circ S = \Phi(S) + A\circ S = x .
$$
As is well-known: if right inverses all exist, then they are
full inverses.  Review of the proof:
Suppose $T \circ S = x$ as found.  Start with $S$ and
get a right-inverse $T_1$ so $S \circ T_1 = x$.
Then
$T = T \circ x
	= T \circ (S \circ T_1)
	= (T \circ S) \circ T_1
	= x \circ T_1 = T_1$.
\end{proof}

\begin{pr}\label{inverse}
The set $\P$ is a group under composition.
\end{pr}
\begin{proof}
Let $T \in \P$.  Let $p = \expo T$,
so that $\log_k \circ T \circ \exp_k \sim \exp_p$
for large enough $k$.
Let $T_1 = \log_k \circ T \circ \exp_{k-p}$, so that
$T_1 \sim x$ and (if $k$ is large enough)
$T_1$ is log-free.  By Proposition~\ref{inverse1}
there is an inverse, say $T_1 \circ S_1 = x$.
Write $S = \exp_{k-p} S_1 \circ \log_k$.
Then $T \circ S = \exp_k\circ T_1 \circ \log_{k-p} 
\circ \exp_{k-p} \circ S_1 \circ \log_k = x$.
\end{proof}

\begin{re}
We need a grid-based version of Proposition~\ref{mvt1} to prove
existence of a grid-based compositional inverse using a grid-based
fixed-point theorem.  This is done in \cite[\Winv]{edgarw}.
\end{re}

\subsection*{An Example Inverse}
Consider the transseries $S = \log x + 1 + x^{-1} \in \P$.
We want to discuss its compositional inverse.  According
to the method above, we should compute
the inverse of $S_1 = S \circ \exp = x + 1 + e^{-x} \in \P$.
And if $T_1 = S_1^{[-1]}$, then $S^{[-1]} = \exp \circ T_1$.

For the inverse of $S_1 = x+1+e^{-x}$, write $A = 1+e^{-x}$
and solve by iteration
$Y = \Phi(Y)$, where $\Phi(Y) = x - A \circ Y = x - 1 - e^{-Y}$.
We end up with
\begin{align*}
	T_1 &= x - 1 - e e^{-x} - e^2 e^{-2x} - \frac{3e^3}{2} e^{-3x}
	-\frac{8e^4}{3} e^{-4x} +\cdots
	\\ &= x - 1 - \sum_{j=1}^\infty a_j e^{-jx}
\end{align*}
either by iteration, or with a linear equation for each
$a_j$ in terms of the previous ones.  (And $a_j$ is rational
times $e^j$.)  And then
\begin{align*}
	S^{[-1]} &= e^{T_1} = \frac{1}{e} e^x - 1 -\frac{e}{2} e^{-x}
	-\frac{2e^2}{3} e^{-2x}-\frac{9e^3}{8} e^{-3x}
	-\frac{32e^4}{15} e^{-4x}+\cdots
	\\ &= \frac{1}{e} e^x - 1 - \sum_{j=1}^\infty b_j e^{-jx} .
\end{align*}

\subsection*{Compositional Equations}
Because of the group property Proposition~\ref{inverse}
(or the grid-based version \cite[\Winv]{edgarw}), we know:
Let $S,T \in \T$.  If $S,T$ are both large and positive, then
there is a unique $Y \in \P$ with $S = T \circ Y$.

\begin{pr}\label{compcases}
Let $S,T \in \T$.  Then there is a unique $Y \in \P$
with $S = T \circ Y$ in each of the following cases:
$S$ and $T$ are both:
\begin{itemize}
\item[{\rm(a)}] large and positive
\item[{\rm(b)}] small and positive
\item[{\rm(c)}] large and negative
\item[{\rm(d)}] small and negative
\item[{\rm(e)}] For some $c \in \R$, $c \ne 0$,
$S \sim c$, $T \sim c$, $S>c$, $T>c$.
\item[{\rm(f)}] For some $c \in \R$, $c \ne 0$,
$S \sim c$, $T \sim c$, $S<c$, $T<c$.
\end{itemize}
There is a nonunique $Y \in \P$ with
$S = T \circ Y$ in case: for some $c \in \R$,
both $S=c$ and $T=c$.  In all other cases,
there is no $Y$ with $S = T \circ Y$.
\end{pr}
\begin{proof}
(a) is from Proposition~\ref{inverse}.
(b)~Apply (a) to $1/S$ and $1/T$.
(c)~Apply (a) to $-S$ and $-T$.
(d)~Apply (b) to $-S$ and $-T$.
(e)~Apply (b) to $S-c$ and $T-c$.
(f)~Apply (d) to $S-c$ and $T-c$.

The concluding cases are clear.
\end{proof}

\subsection*{Mean Value Theorem}
Using Proposition~\ref{posderiv}, we get a MVT.

\begin{pr}\label{mvt3}
Given $A \in \T, S_1, S_2 \in \LP$, $S_1 < S_2$,
there is $S \in \LP$ so that
$$
	\frac{A \circ S_2 - A \circ S_1}{S_2-S_1} = A' \circ S .
$$
\end{pr}
\begin{proof}
Write $B = (A \circ S_2 - A \circ S_1)/(S_2-S_1)$.
We claim that Proposition~\ref{compcases} shows that there
is a solution $S$ to $B = A'\circ S$.  So we have
to show that $A', B$ are in the same case of
Proposition~\ref{compcases}.

Let $c \in \R$.  If
$A' > c$, then $(A - cx)'>0$, and therefore
by Proposition~\ref{posderiv}
$(A -cx)\circ S_1 < (A -cx)\circ S_2$, so
$A\circ S_2 - A\circ S_1 > c(S_2-S_1)$, so
$(A\circ S_2 - A\circ S_1)/(S_2-S_1) > c$,
so $B > c$.  Similarly:
if $A' < c$, then $B < c$.
These hold for all real $c$, so in fact
$A'$ and $B$ are in the same case.
\end{proof}

The following proposition, too, has---so far---only an
involved proof, which will not be given here.  See
Section~\ref{taylorsection} for this and
still more versions of the Mean Value Theorem.

\begin{pr}\label{mvt2}
Let $A \in \T, S_1, S_2 \in \LP$.
If $A'' > 0$ and $S_1 < S_2$, then
$$
	A' \circ S_1 <
	\frac{A \circ S_2-A \circ S_1}{S_2-S_1} <
	A'\circ S_2 .
$$
\end{pr}

Using this, we can improve the Mean Value Theorem~\ref{mvt3}:

\begin{pr}\label{mvtbetween}
Given $A \in \T, S_1, S_2 \in \LP$, $S_1 < S_2$,
there is $S \in \LP$, $S_1 < S < S_2$ so that
$$
	\frac{A \circ S_2 - A \circ S_1}{S_2-S_1} = A' \circ S .
$$
\end{pr}
\begin{proof}
First assume $A''>0$.
Let $S$ be as in Proposition~\ref{mvt3}.  By
Proposition~\ref{mvt2},
$A'(S_1) < A'(S) < A'(S_2)$.  So by
Proposition~\ref{posderiv} we conclude $S_1 < S < S_2$.

The case $A''<0$ is similar.  The case $A''=0$
is easy.
\end{proof}

\subsection*{Intermediate Value Theorem}

\begin{pr}\label{ivt}
Let $K,T \in \T$, $A,B \in \LP$.  Assume
$T(A) \le K \le T(B)$.  Then there is $S \in \LP$
with $T(S) = K$ and either $A \le S \le B$ or $A \ge S \ge B$.
\end{pr}
\begin{proof}
If $T(A)=K$, choose $S=A$; if $T(B)=K$, choose $S=B$.
So we may assume $T(A) < K < T(B)$.
We will consider cases for $T$.

(a) First assume $T$ is large and positive.  Then the inverse
$T^{[-1]}$ exists in $\LP$.  Also $T(A),T(B)$ are large and positive,
so $K$, which is between them, is large and positive.
Define $S = T^{[-1]}(K)$.  Of course $T(S) = K$.
Since $T^{[-1]}$ is large and positive
it is increasing (by Proposition~\ref{posderiv}), so applying $T^{[-1]}$ to
$T(A) < K < T(B)$ we get $A < S < B$.

(b)  Assume $T$ is large and negative.  Apply case (a) to $-T$.

(c)  Assume $T$ is small and positive.  Apply case (a) to $1/T$.

(d)  Assume $T$ is small and negative.  Apply case (c) to $-T$.

(e)  Assume there is $a \in \R$ with $T \sim a, T > a$.
Apply case (c) to $T-a$.

(f)  Assume there is $a \in \R$ with $T \sim a, T < a$.
Apply case (d) to $T-a$.

(g)  The only case left is $T=a$ for some $a \in \R$, so
$T(A) = T(B) = a = K$, and this case was
taken care of at the beginning of the proof.  Or
let $S = (A+B)/2$ to get $S$ strictly between $A$ and $B$
when $A \ne B$.
\end{proof}

\begin{re}
Using \ref{ivt} we can deduce \ref{mvtbetween} from \ref{mvt2}
without the need of \ref{mvt3}.  But \ref{mvt2} is
still the difficult step.
\end{re}

\section{Taylor's Theorem}\label{taylorsection}
Here we will formulate many versions of Taylor's Theorem.
Unfortunately, proofs are (as far as I know) still
quite involved.  Proofs (for most cases) will not be included here.
See \cite[\S6]{DMM}
for well-based transseries and \cite[\S5.3]{hoeven}
for grid-based transseries.  But in some cases it may
not be clear that they have proved everything listed here.

Recall definitions $\G_N$, $\G_{N,M}$, $\G_\bullet$, etc.
If $\AA$ is a set of monomials, and $S \in \P$,
write $\AA \circ S := \SET{\g\circ S}{\g \in \AA}$.
Let $U \in \T$, then we say
$U \fst \AA$ if $U \fst \g$ for all $\g \in \AA$.
Recall that if $\g \in \G_{N,M} \setminus \G_{N-1,M}$
and $\g\fst 1$, then $\g \fst \G_{N-1,M}$.

Let $T \in \T$, $S_1, S_2 \in \LP$.  For $n \in \N$ define
\begin{equation*}
	\Delta_n(T,S_1,S_2) :=
	T(S_2) - \sum_{k=0}^{n-1} \frac{T^{(k)}(S_1)}{k!}\,(S_2-S_1)^k .
\end{equation*}
When $S_1, S_2$ are understood, write $\Delta_n(T)$.  The first
few cases:
\begin{align*}
	\Delta_0(T) &= T(S_2),\\
	\Delta_1(T) &= T(S_2)-T(S_1),\\
	\Delta_2(T) &= T(S_2)-T(S_1)-T'(S_1)\cdot(S_2-S_1),\\
	\Delta_3(T) &= T(S_2)-T(S_1)-T'(S_1)\cdot(S_2-S_1)
	- \frac{1}{2}T''(S_1)\cdot(S_2-S_1)^2.\\
\end{align*}

Note that derivatives $\partial^k$ are strongly additive, and
therefore these $\Delta_n$ are also.  That is:
if $S = \sum_{i \in I} A_i$ (in the asymptotic topology),
then $\Delta_n(S) = \sum \Delta_n(A_i)$.

\begin{no} Formulations.\label{taylorlabel}
\begin{itemize}
\item[{[$\A_n$]}] Let $T \in \T_{N,M}$, $T \notin \R$, $S_1, S_2 \in \LP$.
If $N = 0$ assume $S_2-S_1 \fst S_1$.  If $N > 0$ assume
$S_2-S_1 \fst \G_{N-1,M}\circ S_1$.
Let $n \in \N$.  If $T^{(n)} \ne 0$, then
$$
	\Delta_n(T) \sim \frac{T^{(n)}(S_1)}{n!}\,(S_2-S_1)^n .
$$

\item[{[$\A_\infty$]}] Let $T \in \T_{N,M}$, $T \notin \R$, $S_1, S_2 \in \LP$.
If $N = 0$ assume $S_2-S_1 \fst S_1$.  If $N > 0$ assume
$S_2-S_1 \fst \G_{N-1,M}\circ S_1$.  Then
$$
	T(S_2) = \sum_{j=0}^\infty \frac{T^{(j)}(S_1)}{j!}\,(S_2-S_1)^j .
$$

\item[{[$\B_n$]}] Let $T \in \T$, let $S_1, S_2 \in \LP$, and let $n \in \N$.
If $T^{(n+1)}>0$ and $S_1 < S_2$, then
$$
	\frac{T^{(n)}(S_1)}{n!}\,(S_2-S_1)^n < \Delta_n(T)
	< \frac{T^{(n)}(S_2)}{n!}\,(S_2-S_1)^n .
$$
Other cases also: If $T^{(n+1)} < 0$, reverse the inequalities.
If $S_1 > S_2$ and $n$ is even, reverse the inequalities.

\item[{[$\CC_n$]}] Let $T \in \T$, let $S_1, S_2 \in \LP$, and let $n \in \N$.
If $T^{(n)} > 0$ and $S_1 < S_2$, then $\Delta_n(T) > 0$.
Other cases also: If $T^{(n)} < 0$, reverse the inequality.
If $S_1 > S_2$ and $n$ is odd, reverse the inequality.

\item[{[$\DD_n$]}] Let $A,B \in \T$, let $S_1, S_2 \in \LP$, and let
$n \in \N$.
If $A^{(n)} \fst B^{(n)}$ then $\Delta_n(A) \fst \Delta_n(B)$.
\end{itemize}
\end{no}

Some beginning cases.
\begin{itemize}
\item[{[$\A_0$]}] If $(S_2-S_1)$ is appropriately small,
then $T(S_2) \sim T(S_1)$.

\item[{[$\A_1$]}] If $(S_2-S_1)$ is appropriately small,
then $T(S_2) -T(S_1) \sim T'(S_1)\cdot(S_2-S_1)$.
Proved in \ref{tterm}.

\item[{[$\B_0$]}] If $T'>0$ and $S_1 < S_2$, then
$T(S_1) < T(S_2) < T(S_2)$. (Second inequality is too strong.)
This is \ref{posderiv}, proved in \ref{posderivthm}.

\item[{[$\B_1$]}] If $T''>0$ and $S_1 \ne S_2$, then
$$
	T'(S_1) < \frac{T(S_2)-T(S_1)}{S_2-S_1}
	< T'(S_2) .
$$
This is \ref{mvt2}.

\item[{[$\CC_0$]}] If $T > 0$, then $T(S_2) > 0$.
This is in \ref{ineq1}.

\item[{[$\CC_1$]}] If $T' > 0$ and $S_1 < S_2$, then $T(S_2)-T(S_1) > 0$.
This is \ref{posderiv} again.

\item[{[$\DD_0$]}] If $A \fst B$ then $A(S_2) \fst B(S_2)$.
This is in \ref{ineq1}.

\item[{[$\DD_1$]}] If $A' \fst B'$ then $A(S_2) - A(S_1) \fst B(S_2)-B(S_1)$.
This is \ref{derivcompare}, proof in \ref{posderivthm}.

\end{itemize}

A variant form of [$\B_n$] follows using the intermediate
value theorem (a consequence of [$\B_1$]).

\begin{itemize}
\item[{[$\B_n'$]}] Let $T \in \T$, let $S_1, S_2 \in \LP$, and let $n \in \N$.
If $S_1 \ne S_2$, then
there exists $\widetilde{S}$ strictly between $S_1$ and $S_2$
such that
$$
	\Delta_n(T,S_1,S_2) = \frac{T^{(n)}\big(\widetilde{S}\big)}{n!}\,(S_2-S_1)^n .
$$
\end{itemize}

\subsection*{Good Proofs Needed---But What Methods?}
A good exposition is needed for the proofs of the principles
stated in \ref{taylorlabel}.  First steps
are seen below (Section~\ref{simpletaylor} for $[\A_1]$
and Section~\ref{involvedproof} for $[\CC_1]$ and $[\DD_1]$).
Now proofs for $[\A_n]$ and $[\A_\infty]$ should be possible along
the same lines.
But I think further proofs for $[\B_n], [\CC_n], [\DD_n]$
along those lines will be
ugly or impossible.  So a better approach is needed.
Even if proofs can, indeed, be found in the literature
(such as \cite[\S6]{DMM} and \cite[\S5.3]{hoeven}), they are
not as elementary as one might hope.

Related results could be expected from the same methods, perhaps.
For example, does the following follow from the principles listed
above, or would it require additional proof?

\textit{Let $U,V \in \T$, $S_1,S_2 \in \LP$.
If $U'>0, V>0, S_1 < S_2$, then}
$$
	U(S_1) \int_{S_1}^{S_2} V <
	\int_{S_1}^{S_2} UV < U(S_2) \int_{S_1}^{S_2} V .
$$
Or: \textit{There exists $\widetilde{S}$ between $S_1$ and $S_2$
with}
$$
	\int_{S_1}^{S_2} UV = U\big(\widetilde{S}\,\big) \int_{S_1}^{S_2} V .
$$
Equivalently: \textit{Let $A,B \in \T, S_1, S_2 \in \LP$
with $B'\ne 0$ and $S_1 \ne S_2$.
Then there exists $\widetilde{S}$ between $S_1$ and $S_2$ with}
$$
	\frac{A(S_2)-A(S_1)}{B(S_2)-B(S_1)} =
	\frac{A'\big(\widetilde{S}\,\big)}{B'\big(\widetilde{S}\,\big)} .
$$
[Equivalence comes from writing $B'=V$, $A'=UV$.]

One method used for proofs such as these (in conventional calculus)
suggests that we need to know about \Def{transseries of two variables}
in order to use the same proofs in this setting.  This remains
to be properly defined and investigated.

\section{Topology and Convergence}\label{sec:converge}
In \cite[\Easympt]{edgar} we defined only the ``asymptotic topology'' for $\T$.
But there are other topologies or types of convergence.  And none
of them has all of the desirable properties.

The \Def{attractive topology} is described by van der Hoeven \cite{hoevenop};
I will use letter H for it, $T_\gamma \ah T$.  For our situation
(with totally ordered valuation group $\G$) it is also the
order topology for $\T$ and the topology arising from the valuation
$\mag$.

\begin{de}
Let $T_\gamma$ be a net in $\T$ and let $T \in \T$.  Then
$T_\gamma \ah T$ iff for every $\m \in \G$ there is $\gamma_\m$
such that for all $\gamma \ge \gamma_\m$ we have
$T - T_\gamma \fst \m$.
\end{de}

This is the convergence of a metric.  Because every transseries
has finite height, there is a countable base for the
H-neighborhoods of zero made up of the sets
$$
	\o(1/\exp_m) = \SET{T \in \T}{T \fst 1/\exp_m}
	\qquad\text{for $m=0, 1, 2, \cdots$}.
$$
Here, as usual, $\exp_0 = x, \exp_1 = e^x, \exp_2 = e^{e^x}$, and so on.

Continuity:  (The ``$\epsilon$--$\delta$'' type definition.)
A function $\Psi \takes \T \to \T$ is H-con\-tin\-uous
at $S_0 \in \T$ iff: for every $\m \in \G$ there is $\n \in \G$
so that for all $S \in \T$, if
$S - S_0 \fst \n$ then $\Psi(S) - \Psi(S_0) \fst \m$.
We may write it like this:
$\Psi\big(S_0 + \o(\n)\big) \subseteq \Psi(S_0) + \o(\m)$.

\medskip
The \Def{asymptotic topology} I get from Costin \cite{costintop};
I will use letter C for it, $T_j \ac T$.  Recall the definition:

\begin{de}
$T_j \overset{\ebmu,\bm}{\longrightarrow} T$ iff
$\supp(T_j) \subseteq \GRID^{\ebmu,\bm}$ for all $j$ and
$\supp(T_j-T)$ is point-finite;

$T_j \muto T$ iff there exists $\bm$ with
$T_j \overset{\ebmu,\bm}{\longrightarrow} T$;

$T_j \ac T$ iff there exists $\bmu$ with $T_j \muto T$;
\end{de}

Sets $\T^{\ebmu,\bm} = \SET{T \in \T}{\supp T \subseteq \GRID^{\ebmu,\bm}}$
are metrizable for $\ac$.
The asymptotic topology for all of $\R\lbb \G \rbb = \T$
is an inductive limit:
open sets are easily described, convergence (except for sequences)
is not.
A set $\SU \subseteq \T$ is C-\Def{open} iff
$\SU \cap \T^{\ebmu,\bm}$ is open in $\T^{\ebmu,\bm}$
(according to $\ac$) for all $\bmu$ and $\bm$.

\begin{de}
Here is a similar convergence, applying to well-based transseries,
but which makes sense even for grid-based transseries.

Let $\AA \subseteq \G$ be well ordered.
$T_j \overset{\AA}{\longrightarrow} T$ iff
$\supp(T_j) \subseteq \AA$ for all $j$ and
$\supp(T_j-T)$ is point-finite;

$T_j \aw T$ iff there exists well ordered $\AA \subseteq \G$
with $T_j \overset{\AA}{\longrightarrow} T$.
\end{de}

Sets $\T_\AA := \SET{T \in \T}{\supp T \subseteq \AA}$
are metrizable for $\aw$, since $\AA$ is countable.
As before, the W-topology for all of $\T$
is an inductive limit:
A set $\SU \subseteq \T$ is W-\Def{open} iff
$\SU \cap \T_\AA$ is open in $\T_\AA$
(according to $\aw$) for all well ordered $\AA$.

\subsection*{Basics}
The attractive topology is discrete on
$\T_{NM} = \R\lbb \G_{NM} \rbb$, the transseries of
given height and depth.  Indeed, if $T \in \T_{NM}$,
then for $n>N$ the set $T+\o(1/\exp_n)$ is open and
$\T_{NM} \cap \big(T+\o(1/\exp_n)\big) = \{T\}$.
So a net contained in some $\T_{NM}$ converges
iff it is eventually constant.
The series representing $T \in \T$
(for example series $\sum_{j=0}^\infty x^{-j}$) is essentially
never H-convergent---it is H-convergent only if it has all
but finitely many terms equal to $0$.

For each $\m$, the ``coefficient'' map $T \mapsto T[\m]$ is continuous from
$(\T,\text{asymptotic})$ to $(\R,\text{discrete})$.
Indeed, given $\m$ and $T_0 \in \T$, the function $T[\m]$ is
constant on the coset $T_0 + \o(\m)$.  So it is better than
continuous: it is locally constant.

The series representing $T \in \T$ is C-convergent to $T$.
And W-convergent.  Consider the sequence $x^{-\log j}$,
($j=1,2,\cdots$).  This set is well ordered but not
grid-based.  So $x^{-\log j} \aw 0$ but not
$x^{-\log j} \ac 0$.

Coefficient maps $T[\m]$ are
C-continuous and W-continuous.
I guess locally constant, too, since
sets of the form
$\SET{T \in \T}{T[\m] = a}$
are C-open and W-open.

The whole transline $\T$ is not metrizable for C or W.
Let $T_{jk} = x^{-j}e^{kx}$.  Then according to C convergence,
$$
	\lim_{j\to\infty} T_{jk} = 0\qquad\text{for each $k \in \N$.}
$$
In a metric space, it would then be possible to choose
$j_1, j_2, j_3, \cdots$ so that
$$
	\lim_{k \to \infty} T_{j_k k} = 0 .
$$
(For example,  for each $k$ choose $j_k$ so that the distance from
$T_{j_k k}$ to $0$ is ${}< 1/k$.)
But that is false for C or W.

\subsection*{Well-Based Pseudo Completeness}
A system $T_\alpha \in \T$, where $\alpha$ ranges over
the ordinals up to some limit ordinal $\lambda$, is
called a \Def{pseudo Cauchy sequence} iff
$T_\alpha - T_\beta \fgt T_\beta - T_\gamma$
for all $\alpha < \beta < \gamma < \lambda$.
And $T$ is a \Def{pseudo limit} of $T_\alpha$
iff $T_\alpha - T \sim T_\alpha - T_{\alpha+1}$
for all $\alpha < \lambda$.
A space is called \Def{pseudo complete}
if every pseudo Cauchy sequence has a pseudo limit.
The well based Hahn sequence spaces $\R[[\M]]$
are pseudo complete.  (Grid based spaces $\R\lbb\M\rbb$
are usually not pseudo complete.  Instead there
is a ``geometric convergence'' explained in \cite[\Wgeomconv]{edgarw}.)
But the transseries field $\T$, a proper subset of $\R[[\G]]$,
is not pseudo complete.

A pseudo limit is not expected to be unique, but in our setting there
is a distinguished pseudo limit.  It is the limit
(in the W topology) of $S_\beta$, where
$S_\beta$ is the longest common truncation of
$\SET{T_\alpha}{\alpha \ge \beta}$.
See the ``stationary limit'' in \cite{hoevenop}.

Here is a well-based fixed point theorem from van der Hoeven
\cite[Thm.~4.7]{hoevenop}.  Note that in our case where
$\M$ is totally ordered, the special ordering
$\;\fst\!\!\!\cdot\;$ coincides with the usual ordering $\;\fst\;$.

\begin{pr}\label{wellfixed}
Let $\Phi \takes \R[[\M]] \to \R[[\M]]$.  Assume
for all $T_1, T_2 \in \R[[\M]]$, if $T_1 \ne T_2$,
then $\Phi(T_1)-\Phi(T_2) \fst T_1-T_2$.
Then there is a unique $S \in \R[[\M]]$ such that $\Phi(S) = S$.
\end{pr}
\begin{proof}
Uniqueness.  Assume $\Phi(S_1) = S_1$ and $\Phi(S_2) = S_2$.
If $S_1 \ne S_2$, then $\Phi(S_1) - \Phi(S_2) = S_1-S_2 \not\fst S_1-S_2$,
a contradiction.  So $S_1 = S_2$.

Existence (outline).  Choose any nonzero $T_0 \in \R[[\M]]$.
For ordinals $\alpha$
we define $T_\alpha$ recursively.  Assume  $T_\alpha$ has been
defined.  Consdier two cases. If
$\Phi(T_\alpha) = T_\alpha$, then $S=T_\alpha$ is
the required result.
Otherwise, let $T_{\alpha+1} = \Phi(T_\alpha)$.  If $\lambda$ is a limit
ordinal, and $T_\alpha$ has been defined for all $\alpha < \lambda$,
then (recursively) $T_\alpha$ is pseudo Cauchy, so let
$T_\lambda$ be a pseudo limit of $(T_\alpha)_{\alpha < \lambda}$.
Eventually the process must end because there are more ordinals
than elements of $\R[[\M]]$.
\end{proof}

Example. Consider $Q = x + \log x + \log_2 x + \log_3 x +\cdots$.
The partial sums constitute a pseudo Cauchy sequence
in $\T$, but
the pseudo limits (such as $Q$ itself) in $\R[[\G]]$
are not in $\T$.  This $Q$ is the solution of
$\Phi(Y) = Y$ where $\Phi(Y) = x + (Y\circ \log)$ is
contracting on $\R[[\G]]$.

\subsection*{Addition}
Addition $(S,T) \mapsto S+T$ is H-continuous.
Given $\m \in \G$, we have
$$
	\big(S + \o(\m)\big) + \big(T + \o(\m)\big) \subseteq
	\big(S+T\big) + \o(\m) .
$$
Addition is C-continuous.  Assume $S_j \ac S$, $T_j \ac T$.
There is $\AA = \GRID^{\ebmu,\bm}$
with $S,T,S_j,T_j \in \T_\AA$.
If $\g \in \AA$,
then for all but finitely many $j$ we have
$S_j[\g]=S[\g]$ and $T_j[\g] = T[\g]$, so that
$(S_j+T_j)[\g] = (S+T)[\g]$.  Thus $S_j+T_j \ac S+T$.
Addition is W-continuous: same proof, except that
$\AA$ is merely required to be well ordered.

\subsection*{Multiplication}
Multiplication $(S,T) \mapsto ST$ is H-continuous.
We have
$$
	\big(S + \o(\m)\big)\;\big(T + \o(\n)\big) \subseteq
	ST + \o\big((\mag S)\n + (\mag T)\m + \m\n\big),
$$
so given $S,T \in \T$ and $\g \in \G$, there exist
$\m, \n \in \G$ with $\big(S + \o(\m)\big)\;\big(T + \o(\n)\big) \subseteq
ST + \o(\g)$.

Multiplication is C-continuous \cite [\Emultcontin]{edgar}.
Let $S_i \ac S$, $T_i \ac T$.
There exist $\bmu, \bm$ so that
$S_i \overset{\ebmu,\bm}{\longrightarrow} S$ and
$T_i \overset{\ebmu,\bm}{\longrightarrow} T$.
Then there exist $\tbmu, \widetilde{\bm}$ with
$\GRID^{\ebmu,\bm} \cdot \GRID^{\ebmu,\bm}
\subseteq \GRID^{\tebmu,\widetilde{\bm}}$.
(In fact we may take $\tbmu = \bmu$ and $\widetilde{\bm} = 2\bm$.)
Now given any $\g \in \GRID^{\tebmu,\widetilde{\bm}}$,
there are finitely many pairs
$(\m,\n) \in \GRID^{\ebmu,\bm}\times\GRID^{\ebmu,\bm}$ with
$\m\n=\g$.  For each such $\m$ or $\n$, except for
finitely many indices $i$ we have $S_i[\m]=S[\m]$ and
$T_i[\n] = T[\n]$.  So, except for $i$
in a finite union of finite sets we have
$(S_iT_i)[\g] = (ST)[\g]$.  Therefore
$S_i T_i \ac ST$.

Multiplication is W-continuous.  This will be similar to C-continuity.
We need to use \cite[\Ewellprod]{edgar}:
Given any well ordered $\AA \subseteq \G$, the set
$\AA\cdot\AA$ is well ordered, and for any
$\g \in \AA \cdot\AA$, there are finitely many pairs
$(\m,\n) \in \AA \times \AA$ with $\m\n=\g$.

\subsection*{Differentiation}
First note
$$
	\big(T + \o(\n)\big)' \subseteq T' + \o(\n')\qquad
	\text{provided $\n \ne 1$.}
$$
Given any $\m \in \G$, there is $S \in \T$ with
$S' = \m$ by \cite[\Eintegral]{edgar}.
We may assume the constant term of $S$ is zero.
So let $\n = \mag(S)$, and then $\n' \sim S' = \m$ so
$$
	\big(T + \o(\n)\big)' \subseteq T' + \o(\m).
$$
In fact, since $\n$ did not depend on $T$, we have
shown that differentiation is H-uniformly continuous.

Now consider C-continuity.

From \cite[\Ederivexist]{edgar} or
\cite[\Wderivconvergence]{edgarw}: Given $\bmu, \bm$, there exist
$\tbmu,\widetilde{\bm}$ so that if
$T \in \T^{\ebmu,\bm}$ then $T' \in \T^{\tebmu,\widetilde{\bm}}$
and if $T_j \in \T^{\ebmu,\bm}$ with
$T_j \overset{\ebmu,\bm}{\longrightarrow} T$, then
$T'_j \overset{\tebmu,\widetilde{\bm}}{\longrightarrow} T'$.

W-continuity probably needs a proof like \cite[\Ederivexist]{edgar}.

The derivative is computed as H-limit:
From \ref{taylorlabel}$[\A_2]$ we have: for $U \fst \G_{N-1,M} \circ S$,
$$
	\frac{T(S+U)-T(S)}{U} - T'(S)\sim \frac{T''(S) U}{2} ,
$$
so in the H-topology
$$
	T'(S) = \lim_{U \to 0} \frac{T(S+U)-T(S)}{U} .
$$

\subsection*{Integration}

Integration is continuous?  This should be investigated.

\subsection*{Composition (Left)}
For a fixed (large positive) $S$, consider the composition function
$T \mapsto T \circ S$.

If $T_i \ac T$, then $T_i \circ S \ac T \circ S$
\cite[\Ecompcontin]{edgar}, which depends on \cite[\Ecompexist]{edgar}.

For W-continuity we need a proof like \cite[\Ecompexist]{edgar}.

Now consider H-continuity.  Note
$$
	\big(T + \o(\n)\big)\circ S \subseteq (T \circ S) + \o(\n \circ S) .
$$
So we need:  Given $\m \in \G$, there is $\n$ such that
$\n \circ S \fsteq \m$.  So we would have H-uniform continuity.
Certainly this is true, since we can take $\n = 1/\exp_N$ for large enough
$N$.  But what about a less drastic solution?  Of course:
$\n = \m \circ S^{[-1]}$.  Or if we insist that $\n$ be
a monomial, $\n = \mag\big(\m \circ S^{[-1]}\big)$.

\subsection*{Composition (Right)}
What about continuity of composition $T \circ S$ as a function of the right
composand $S$?  It is certainly false for C and W convergence.
Indeed, let $T = e^x$.  Then to
compute even one term of $e^S$ we need to know
all of the large terms of $S$; there could be infinitely many
large terms.

\medskip Now consider H-continuity.

\begin{pr}
{\rm(i)} Function $\exp$ is $\mathrm{H}$-continuous on $\T$.
{\rm(ii)}~Function $\log$ is $\mathrm{H}$-continuous on
{\rm(}the positive subset
of{\,\rm)} $\T$.
{\rm(iii)}~Let $T \in \T$.  Then function $S \mapsto T \circ S$
is $\mathrm{H}$-continuous on $\P$.
\end{pr}
\begin{proof}
(i)
Let $S_0 \in \T$ and $\m \in \G$ be
given.  Let
$$
	\n = \begin{cases}
	\m\mag(e^{-S_0}), &\text{if } \m\mag(e^{-S_0}) \fsteq 1 ,
	\\
	1, &\text{otherwise.}
	\end{cases}
$$
Now if $s := S - S_0 \fst \n$, we have $s \fst 1$ so
$e^s-1\sim s \fst \n$.  And
$$
	e^S - e^{S_0} = e^{S_0}(e^{S-S_0} - 1)
	\fst e^{S_0} \n \fsteq \m .
$$
That is: if $S \in S_0 + \o(\n)$, then
$e^S \in e^{S_0} + \o(\m)$.  This shows that $\exp$
is H-continuous at $S_0$.

(ii)~Let $S_0 > 0$ and $\m \in \G$ be given.  Then take
$$
	\n = \begin{cases}
	\m \mag S_0, &\text{if } \m \fsteq 1,
	\\
	\mag S_0, &\text{otherwise.}
	\end{cases}
$$
Now assume $S - S_0 \fst \n$.  Then
$$
	\frac{S-S_0}{S_0} \fst \frac{\n}{\mag S_0} \fsteq 1
$$
so
$$
	\log(S) - \log(S_0) = \log\frac{S}{S_0}
	=\log\left(1+\frac{S-S_0}{S_0}\right)
	\sim \frac{S-S_0}{S_0} \fst \frac{\n}{\mag S_0} \fsteq \m .
$$

(iii)
We will apply Corollary~\ref{inductivecor}.  Let
$\RR$ be the set of all $T \in \T$ such that the function
$S \mapsto T \circ S$ is H-continuous.  We now check
the conditions of Corollary~\ref{inductivecor}.
If $\g \in \RR$, then $\g \circ \log \in \RR$ by (ii);
this proves (f$\,'$).  If $L \in \RR$, then
$e^L \in \RR$ by (i).  And $x^b = e^{b\log x} \in \RR$
by (i) and (ii).  So $x^b e^L \in \RR$.
This proves (e$'$).

Finally we must prove (d$'$).
Let $T \in \T$ and assume $\supp T \subseteq \RR$.
(If $T=0$ we have $T \in \RR$ trivially, so assume $T \ne 0$.)
Let $\g_0 = \mag T$, so $\g_0 \in \RR$.  Note that
$T/\g_0 \fe 1 \fst x$.  By Proposition~\ref{mvt1} we have
$$
	\frac{T}{\g_0} \circ S_2 - \frac{T}{\g_0} \circ S_1
	\fst S_2 - S_1,
$$
so $S \mapsto (T/\g_0)\circ S$ is (uniformly) H-continuous.
By hypothesis, $S \mapsto \g_0 \circ S$ is H-continuous.
So (since multiplication is H-continuous) it follows that
the product
$$
	S \mapsto \bigg(\frac{T}{\g_0} \circ S\bigg) \cdot
	\bigg(\g_0 \circ S\bigg)
	= T \circ S
$$
is H-continuous.

So we may conclude $\RR = \T$ as required.
\end{proof}

\subsection*{Fixed Point}
Fixed point with parameter: conditions on
$\Phi(S,T)$ beyond ``contractive in $S$ for each $T$\,'' so that if
$S=S_T$ solves $S = \Phi(S,T)$, then $T \mapsto S_T$ is a
continuous function of $T$. Compare \cite{hoevenop}.
This should be investigated for all three topologies.

\section{Proof for the Simplest Taylor Theorem}\label{simpletaylor}
I said in Section~\ref{taylorsection} that proofs for
Taylor's Theorem are quite involved.  Here I include a
proof for the simplest one, namely \ref{taylorlabel}[{$\A_1$}].

\begin{pr}\label{tterm}
Let $T \in \T_{N,M}$, $T \not \in \R$, $S \in \P$, $U \in \T$.
If $N=0$, assume $U \fst S$.  If $N>0$, assume
$U \fst \G_{N-1,M}\circ S$.  Then
\begin{equation*}
	T(S+U) - T(S) \sim T'(S) \cdot U .
\tag{$\dagger$}
\end{equation*}
\end{pr}
\begin{proof}
For $N,M \in \N$, let $\A(N,M)$ mean that
the statement of the theorem holds for all $T \in \G_{N,M}$,
and let $\B(N,M)$ mean that the
the statement of the theorem holds for all $T \in \T_{N,M}$.
Note for any $N,M \in \N$, from $U \fst \G_{N,M}\circ S$ it follows that
$U \fst S$: Indeed, $1 \in \G_{N,M}$, so
$U \fst 1 \fst S$.

(1) Claim: Let $S \in \P$, $U \in \T$, and assume
$U \fst S$.  Then
\begin{equation*}
	\log(S+U) - \log(S) \sim \frac{U}{S} .
\tag{$\dagger\log$}
\end{equation*}

Indeed, $U/S \fst 1$, so by the Maclaurin series for
$\log(1+z)$ we get
\begin{align*}
	\log(S+U) &= \log\left(S\left(1+\frac{U}{S}\right)\right)
	=
	\log(S) + \log\left(1+\frac{U}{S}\right)
	\\ &=
	\log(S) - \sum_{j=1}^\infty \frac{(-1)^j}{j}\left(\frac{U}{S}\right)^j
	= \log(S) +\frac{U}{S} + \o\left(\frac{U}{S}\right) .
\end{align*}

(2) $\A(0,0)$:  Let $b \in \R$, $b \ne 0$, $S \in \P$, $U \in \T$, and
assume $U \fst S$.  Then
\begin{equation*}
	(S+U)^b - S^b \sim b S^{b-1} \cdot U .
\tag{$\dagger \G_0$}
\end{equation*}

Now $U/S \fst 1$, so by Newton's binomial series we get
\begin{align*}
	(S+U)^b &= S^b \left(1+\frac{U}{S}\right)^b
	= S^b \sum_{j=0}^\infty \binom{b}{j}\left(\frac{U}{S}\right)^j
	\\ &=
	S^b\left(1+b\frac{U}{S} + \o\left(\frac{U}{S}\right)\right)
	= S^b + bS^{b-1}\cdot U + \o\left(S^{b-1}\cdot U\right) .
\end{align*}
Note that even if $b=0$ the equation
$(S+U)^b=S^b+bS^{b-1}U+\o(S^{b-1}U)$ remains true.

(3) $\B(0,0)$: Let $T \in \T_0$, $T \not\in \R$, $S \in \P$, $U \in \T$,
and assume $U \fst S$.  Then ($\dagger$).

Let $\dom T = a_0 x^{b_0}$.  First consider the case $b_0 \ne 0$.
Then $T' \sim a_0b_0x^{b_0-1}$ and
$$
	a_0(S+U)^{b_0} - a_0S^{b_0} = a_0b_0S^{b_0-1}\cdot U + \o(S^{b_0-1}\cdot U)
	= T'(S)\cdot U + \o(T'(S)\cdot U) .
$$
For any other term
$a x^b$ of $T$, we have $b<b_0$ and
$$
	a(S+U)^b - aS^b = abS^{b-1}\cdot U + \o(S^{b-1}\cdot U)
	= \o(S^{b_0-1}\cdot U) = \o(T'(S)\cdot U) .
$$
Summing all the terms of $T$, we get
$$
	T(S+U) - T(S) = T'(S)\cdot U + \o(T'(S)\cdot U) .
$$
Now take the case $b_0=0$.  Subtract the dominance: $T_1 = T-a_0$.
Since we assumed $T \not\in\R$, it follows that
$T_1 \ne 0$.  Also
$T' = T_1'$.  Applying the previous case to $T_1$, we get
\begin{align*}
	T(S+U)-T(S) &= 0+T_1(S+U)-T_1(S)
	= T_1'(S)\cdot U + \o(T_1'(S)\cdot U)
	\\ &= T'(S)\cdot U + \o(T'(S)\cdot U)  .
\end{align*}

(4) Let $N \ge 0$.  Claim: If $\B(N,0)$, then $\A(N+1,0)$.

Assume $\B(N,0)$.
Let $T \in \G_{N+1}$, $T \ne 1$.  Then $T = e^L$, where $L\ne 0$ is
purely large in $\R\lbb\G_N\cup\{\log x\}\rbb$.
Let $S \in \P$, and let $U \in \T$ with
$U \fst \G_{N}\circ S$.  Now in particular,
$U \fst \G_{N-1}\circ S$ if $N>0$ or $U \fst S$ if $N=0$,
so $L(S+U)-L(S) \sim L'(S)\cdot U$.  But also
$L' \in \T_N$ [noting that $(\log x)' = 1/x \in \T_N$] and
$L' \ne 0$, so $1/L' \in \T_N$ and thus
$\mag(1/L') \in \G_N$.  From the assumption
$U \fst \G_N\circ S$ we get $U \fst 1/L'(S)$,
so $L'(S)\cdot U \fst 1$.
So
$$
	U_1 := L(S+U)-L(S) \sim L'(S)\cdot U \fst 1.
$$
Therefore we may use the Maclaurin
series for $e^z$ to expand:
\begin{align*}
	T(S+U)-T(S) &=
	e^{L(S+U)}-e^{L(S)}
	= (e^{U_1}-1) e^{L(S)}
	=
	\big(U_1 + \o(U_1)\big) e^{L(S)}
	\\ &=
	\big(L'(S)\cdot U + \o(L'(S)\cdot U)\big) e^{L(S)}
	=
	T'(S)\cdot U + \o(T'(S)\cdot U) .
\end{align*}

(5) Let $N \ge 1$.  Claim:  If $\A(N,0)$ then $\B(N,0)$.

Same argument as (3).

(6) Let $M \in \N$.  Claim: If $\B(0,M)$ then
$\B(0,M+1)$.

Assume $\B(0,M)$.
Let $T \in \T_{0,M+1}$, $T \not\in \R$,
$S \in \P$, $U \in \T$, and assume
$U \fst S$.
Then $T = T_1 \circ \log$, with $T_1 \in \T_{0,M}$,
and $T'(x) = T_1'(\log x)/x$.  Now by (1), 
$$
	U_1 := \log(S+U) - \log(S) \sim \frac{U}{S} \fst 1 \fst \log S.
$$
Now applying $\B(0,M)$ to $T_1, S_1 = \log S, U_1$, we get
\begin{align*}
	T(S) - T(S+U) &=
	T_1(\log(S+U)) - T_1(\log S)
	=
	T_1(\log S + U_1) - T_1(\log S)
	\\ &=
	T_1(S_1+U_1) - T_1(S_1)
	\sim
	T_1'(S_1) \cdot U_1
	\\ &\sim
	T_1'(\log S) \cdot U/S
	= T'(S) \cdot U .
\end{align*}

(7) Let $N,M \in \N$, $N>0$.  Claim: If $\B(N,M)$ then $\B(N,M+1)$.

Assume $\B(N,M)$.
Let $T \in \T_{N,M+1}$, $T \not\in \R$,
$S \in \P$, $U \in \T$, and assume
$U \fst \G_{N-1,M+1}\circ S$.
Then $T = T_1 \circ \log$, with $T_1 \in \T_{N,M}$,
and $T'(x) = T_1'(\log x)/x$.  Now for any $N,M$ we have
$U \fst S$, so by (1), 
$$
	U_1 := \log(S+U) - \log(S) \sim \frac{U}{S} \fst 1.
$$
Now if we write $S_1 = \log S$, then
$$
	U_1 \sim \frac{U}{S} \fst U \fst
	\G_{N-1,M+1}\circ S = \G_{N-1,M} \circ S_1 .
$$
Applying $\B(N,M)$ to $T_1, S_1, U_1$, we get
\begin{align*}
	T(S) - T(S+U) &=
	T_1(\log(S+U)) - T_1(\log S)
	=
	T_1(\log S + U_1) - T_1(\log S)
	\\ &=
	T_1(S_1+U_1) - T_1(S_1)
	\sim
	T_1'(S_1) \cdot U_1
	\\ &\sim
	T_1'(\log S) \cdot U/S
	= T'(S) \cdot U .
\end{align*}

(8) By induction we have: $\B(N,M)$ for all $N,M$.
\end{proof}

The other cases \ref{taylorlabel}[{$\A_n$}] and [{$\A_\infty$}]
would be proved in the same way.
See \cite[Sect.~6.8]{DMM}, \cite[Prop.~5.11]{hoeven}.
The argument will perhaps use the formula for the $j$th derivative of a
composite function.

The condition $U \fst \G_{N-1,M}\circ S$ comes from
\cite[Sect.~6.8]{DMM}.  In \cite[Prop.~5.11]{hoeven} we can see
that in fact we do not need to use all of $\G_{N-1,M}$; in
the notation of \cite[\Wtsupp]{edgarw}, it suffices that $U \fst (1/\m)\circ S$
for all $\m \in \mathrm{tsupp}\; T$.

\section{Proof for Propositions \ref{posderiv}
and \ref{derivcompare}}\label{involvedproof}

\begin{de}
Let $\RR \subseteq \T$.  We say $\RR$ satisfies $\CC$
iff for all $T \in \RR$
and all $S_1, S_2 \in \P$ with $S_1 < S_2$,
\begin{align*}
   T' > 0 \Longrightarrow{}& T\circ S_1 < T \circ S_2,
\\ T' = 0 \Longrightarrow{}& T\circ S_1 = T \circ S_2,
\\ T' < 0 \Longrightarrow{}& T\circ S_1 > T\circ S_2 .
\end{align*}
We say $\RR$ satisfies $\DD$ iff for all $A, B \in \RR$,
and all $S_1, S_2 \in \P$ with $S_1 < S_2$,
if $A' \fst B'$, then
\begin{equation*}
	A \circ S_2 - A \circ S_1 \fst 
	B \circ S_2 - B \circ S_1 .
\end{equation*}
\end{de}

So Proposition \ref{posderiv} says $\T$ satisfies
$\CC$ and Proposition \ref{derivcompare} says $\T$ satisfies $\DD$.
These are what I attempt to prove next.
We will use notation
$\T_\AA = \SET{T \in \T}{\supp T \subseteq \AA}$.

\begin{re}\label{ppa}
Let $\RR \subseteq \T$.
$\RR$ satisfies $\CC$ iff $\{T\}$ satisfies $\CC$
for all $T \in \RR$.
$\RR$ satisfies $\DD$ iff $\{A,B\}$ satisfies $\DD$
for all $A,B \in \RR$.  If $\RR$ satisfies
$\CC$, then $\RR \cup \{1\}$ satisfies $\CC$.
If $\RR$ satisfies
$\DD$, then $\RR \cup \{1\}$ satisfies $\DD$.
\end{re}

\begin{lem}\label{ppb}
Let $\AA \subseteq \G$.  If $\AA$ satisfies
$\DD$, then $\T_\AA$ satisfies $\DD$.
\end{lem}
\begin{proof}
Assume $\AA$ satisfies $\DD$.
We may assume $1 \in \AA$. 
Let $A,B \in \T_\AA$ with $A'\fst B'$
and let $S_1, S_2 \in \P$ with $S_1 < S_2$.
If $B$ is replaced by $B-c$ and/or
$A$ is replaced by $A-c$, then both the
hypothesis  $A' \fst B'$ and the conclusion
$A \circ S_2 - A \circ S_1 \fst 
B \circ S_2 - B \circ S_1$ are unchanged.  So we
may assume $A,B$ have no constant terms.
This means $A \fst B$.
Let $\dom B = a_0\g_0$, $a_0 \in \R$, $a_0 \ne 0$, $\g_0 \in \AA$.
Then all terms of $A$ and all terms of $B$ except
for the single term $a_0\g_0$ are ${}\fst \g_0$.
Let $a\g$ be such a term, $a \in \R$, $\g \in \AA$.
Since $\AA$ satisfies $\DD$,
$$
	\g\circ S_2 - \g\circ S_1 \fst
	\g_0\circ S_2 - \g_0\circ S_1
$$
so
\begin{equation*}
	a\g\circ S_2 - a\g\circ S_1 \fst
	\g_0\circ S_2 - \g_0\circ S_1 .
\tag{1}
\end{equation*}
Summing (1) over all terms of $A$, we get
\begin{equation*}
	A\circ S_2 - A\circ S_1 \fst
	\g_0\circ S_2 - \g_0\circ S_1 .
\end{equation*}
Summing (1) over all terms of $B$ except the dominant term,
we get
\begin{equation*}
	B\circ S_2 - B\circ S_1 \asymp
	\g_0\circ S_2 - \g_0\circ S_1 .
\end{equation*}
Therefore, $A\circ S_2 - A\circ S_1 \fst
B\circ S_2 - B\circ S_1$, as required.
\end{proof}

\begin{lem}\label{ppc}
Let $\AA \subseteq \G$.  If $\AA$ satisfies
$\CC$ and $\DD$, then $\T_\AA$ satisfies $\CC$.
\end{lem}
\begin{proof}
Assume $\AA$ satisfies $\CC$ and $\DD$.
We may assume $1 \in \AA$.
Let $T \in \T_\AA$ and let $S_1, S_2 \in \P$
with $S_1 < S_2$.  Since we may replace
$T$ by $T-c$, we may assume $T$ has no constant term.
Let $\dom T = a_0 \g_0$.  Then $T' \sim a_0 \g_0'$,
$\g_0'\ne 0$, so $T'$ has the same sign as $a_0\g_0'$.
We may replace $T$ by $-T$, so it suffices to consider the case
$T'>0$.  Now $\g_0 \in \AA$, which satisfies $\CC$, so
$a_0 \g_0 \circ S_1 < a_0\g_0\circ S_2$.  For all terms
$a\g$ of $T$ other than $a_0\g_0$, we have
$a\g\circ S_2 - a\g\circ S_1 \fst
a_0\g_0\circ S_2 - a_0\g_0\circ S_1$
since $\AA$ satisfies $\DD$.  Summing these terms, we get
$T\circ S_2 - T\circ S_1 \sim
a_0\g_0\circ S_2 - a_0\g_0\circ S_1 > 0$, so
$T\circ S_2 - T\circ S_1>0$ as required.
\end{proof}

\begin{lem}\label{ppd}
$\G_0 \cup \{\log x\}$ satisfies $\CC$.
\end{lem}
\begin{proof}
This is Proposition~\ref{ineq} (a)(b)(c).
\end{proof}

\begin{lem}\label{ppl}
Let $\AA \subseteq \G$.  If $\AA$ satisfies $\CC$,
then $\AA \cup \{\log\}$ satisfies $\CC$
\end{lem}
\begin{proof}
As noted in Lemma~\ref{ppd}, $\{\log\}$ satisfies $\CC$.
Apply Remark~\ref{ppa}.
\end{proof}

\begin{lem}\label{ppe}
$\G_0 \cup \{\log x\}$ satisfies $\DD$.
\end{lem}
\begin{proof}
Let $A, B \in \G_0 \cup \{\log x\}$ with
$A' \fst B'$ and let $S_1, S_2 \in \P$ with
$S_1 < S_2$.  [Since $B=1$ is impossible and
$A=1$ is clear, assume both are not $1$.]
First consider $A = x^a, B = x^b$, so
$A' \fst B'$ means $a<b$.  We must show
$S_2^a-S_1^a \fst S_2^b-S_1^b$.  Write
$S_2=S_1+U$, $U>0$, and consider three cases:
$U \fst S_1$, $U \asymp S_1$, $U \fgt S_1$.

Case $U \fst S_1$.  Then $U/S_1 \fst 1$ and
$$
	S_2^b-S_1^b
	= S_1^b\left[\left(1+\frac{U}{S_1}\right)^b-1\right]
	\sim S_1^b\left[\left(1+\frac{bU}{S_1}\right)-1\right]
	= bS_1^{b-1}U \asymp S_1^{b-1}U .
$$
So $S_2^b-S_1^b \asymp S_1^{b-1}U \fgt S_1^{a-1}U \asymp
S_2^a-S_1^a$.

Case $U \asymp S_1$.  Say $U/S_1 \sim c$, $c \in \R$, $c>0$.
Note $(1+c)^b-1$ is a nonzero constant, so
$$
	S_2^b-S_1^b =
	S_1^b\left[\left(1+\frac{U}{S_1}\right)^b-1\right]
	\sim S_1^b\left[\left(1+c\right)^b-1\right]
	\asymp S_1^b .
$$
So $S_2^b-S_1^b \asymp S_1^b \fgt S_1^a \asymp S_2^a-S_1^a$.

Case $U \fgt S_1$.  Then $S_2 = S_1+U \sim U \fgt S_1$.
If $b>0$, then $S_1^b \fst S_2^b$, so $S_2^b-S_1^b \sim S_2^b$.
But if $b<0$, then $S_1^b \fgt S_2^b$, so
$S_2^b-S_1^b \sim -S_1^b$.  So we may compute:
\begin{align*}
	\text{if $b>a>0$, then }& S_2^b-S^1_b \sim S_2^b
	\fgt S_2^a \sim S_2^a-S_1^a ,
	\\
	\text{if $b>0>a$, then }& S_2^b-S_1^b \sim S_2^b \fgt 1 \fgt
	S_1^a\sim S_1^a-S_2^a ,
	\\
	\text{if $0>b>a$, then }& S_1^b-S_2^b \sim S_1^b \fgt S_1^a
	\sim S_1^a-S_2^a .
\end{align*}

This completes the proof for $x^a \fst x^b$.
The computations for $\log x \fst x^b$ or
$x^a \fst \log x$ are next.

Case $U \fst S_1$.  Then
$$
	\frac{S_2}{S_1} = \frac{S_1+U}{S_1} = 1+\frac{U}{S_1},\qquad
	\log(S_2)-\log(S_1) = \log\frac{S_2}{S_1} \sim \frac{U}{S_1} .
$$
If $b>0$ then $S_2^b-S_1^b \asymp S_1^{b-1} U \fgt U/S_1
\sim \log(S_2)-\log(S_1)$.  And if
$a<0$ then $S_1^a-S_2^a \asymp S_1^{a-1}U \fst U/S_1
\sim \log(S_2)-\log(S_1)$.

Case $U \asymp S_1$.  Then $U/S_1 \sim c$ so
$$
	\log(S_2)-\log(S_1) = \log\left(1+\frac{U}{S_1}\right)
	\sim \log(1+c) \asymp 1 .
$$
If $b>0$, then $S_2^b-S_1^b \asymp S_1^b \fgt 1 \asymp \log(S_2)-\log(S_1)$.
If $a<0$, then $S_1^a-S_2^a \asymp S_1^a \fst 1 \asymp \log(S_2)-\log(S_1)$.

Case $U \fgt S_1$.  Then $S_2/S_1 \fgt 1$ so
$\log(S_2)-\log(S_1) \fgteq \log(S_2)$.  If $b>0$, then
$S_2^b-S_1^b \asymp S_2^b \fgt \log(S_2) \fgteq \log(S_2)-\log(S_1)$.
If $a<0$, then
$S_1^a-S_2^a \asymp S_1^a \fst 1 \fsteq \log(S_2/S_1)
= \log(S_2)-\log(S_1)$.
\end{proof}

\begin{lem}\label{pph}
Suppose $\G_0 \subseteq \AA \subseteq \G_\bullet$
and $\AA$ satisfies $\DD$.  Then
$\AA \cup \{\log x\}$ satisfies $\DD$.
\end{lem}
\begin{proof}
Let $\AA$ satisfy $\DD$, where
$\G_0 \subseteq \AA \subseteq \G_\bullet$.
Let $\fa, \fb \in \AA \cup \{\log x\}$ with
$\fa' \fst \fb'$ and let $S_1, S_2 \in \P$ with
$S_1 < S_2$.  Since $\AA$ already satisfies
$\DD$, we are left only with the two cases $\fa=\log x$
and $\fb = \log x$.  Suppose $\fa = \log x$, so that
$\fb \fgt \log x \fgt 1$.  Since $\fb$ is log-free,
by \cite[\Elogfreepower]{edgar}
there is a real constant $c>0$ with $x^c \fst \fb$.
But $x^c \in \AA$, so
$x^c \circ S_2 - x^c \circ S_1 \fst
\fb \circ S_2 - \fb \circ S_1$.
By Lemma~\ref{ppe} we have
$\log \circ S_2 - \log \circ S_1 \fst
x^c \circ S_2 - x^c \circ S_1$.
Combining these, we get
$\log \circ S_2 - \log \circ S_1 \fst
\fb \circ S_2 - \fb \circ S_1$.

Consider the other case, $\fb = \log x$.
If $\fa = 1$, the conclusion is clear.  If
$\fa \fst \log x$ is log-free and not $1$,
then there is a real constant $c<0$ with
$\fa \fst x^c$.  Then, as in the previous case, we have
$x^c \circ S_2 - x^c \circ S_1 \fgt
\fa \circ S_2 - \fa \circ S_1$ and
$\log \circ S_2 - \log \circ S_1 \fgt
x^c \circ S_2 - x^c \circ S_1$, so
$\log \circ S_2 - \log \circ S_1 \fgt
\fa \circ S_2 - \fa \circ S_1$.
\end{proof}

\begin{lem}\label{ppf}
Suppose $\G_0 \subseteq \AA \subseteq \G_\bullet$.
If $\AA$ satisfies $\CC$ and $\DD$, then
$$
	\widetilde{\AA} :=
	\SET{x^be^L}{b \in \R, L \in \T_\AA\text{ purely large}}
$$
satisfies $\CC$.
\end{lem}
\begin{proof}
First $\AA \cup \{\log\}$ satisfies $\CC$ by Lemma~\ref{ppl}
and $\DD$ by Lemma~\ref{pph}.  Then
$\T_{\AA\cup\{\log\}}$ satisfies $\CC$ by Lemma~\ref{ppc}.

Let $\g \in \widetilde{\AA}$, so $\g = e^L$ with
$L \in \T_{\AA\cup\{\log\}}$ purely large
and let $S_1, S_2 \in \P$ with $S_1 < S_2$.
Then $\g' = L' e^L$ so $\g'$ has the same sign as $L'$.
Take the case $\g'>0$.
Since $L \in \T_{\AA\cup\{\log\}}$ which satisfies $\CC$, we have
$L \circ S_1 < L \circ S_2$.  Exponentiate to get
$\g \circ S_1 < \g \circ S_2$, as required.

The case $\g'<0$ is done in the same way.
\end{proof}

\begin{lem}\label{ppp}
Assume $\T_{\G_N \cup \{\log\}}$ satisfies $\CC$ and $\DD$.
Let $B, L \in \T_{\G_N \cup \{\log\}}$,
with $L$ purely large, and $\fa = e^L \in \G_{N+1}$.
Assume $\fa \fst 1 \fst B$.  Let $S_1, S_2 \in \P$
with $S_1 < S_2$.  Then
$$
	B(S_2) - B(S_1) \fgt \fa(S_1) - \fa(S_2) .
$$
\end{lem}
\begin{proof}
If $L \in \T_{\G_{N-1}\cup\{\log\}}$, then $\fa \in \G_N$,
and this is known by $\DD$.  So assume
$L \not\in \T_{\G_{N-1}\cup\{\log\}}$.
So $\mag L \in \G_N \setminus \G_{N-1}$ has exact height $N$.
Since both hypothesis and conclusion are unchanged when
$B$ is replaced by $-B$, we may assume $B>0$.
Then, since $B$ is large and positive, we also have $B' > 0$.

There are two cases, depending on the size of $S_2-S_1$.

\textit{Case 1.}  $S_2-S_1 \not\fst \G_{N}\circ S_1$.
Let $V = (xe^L/B')\circ S_1$.  Then $V>0$ and
since $B' \in \T_N$ is log-free, and $\mag L$ has
exact height $N$, by \cite[\Eheightwins]{edgar} we have
$x e^L/B' \fst \G_N$, so
$V \fst \G_{N}\circ S_1$.  So $0 < V < S_2-S_1$,
$S_1 < S_1+V < S_2$.  Also
$B'(S_1)\cdot V = S_1 e^{L(S_1)} \fgt e^{L(S_1)}$.
By $\CC$ for $B$, we have $B(S_1+V) < B(S_2)$
and thus
\begin{align*}
	B(S_2)-B(S_1) &> B(S_1+V)-B(S_1)
	\sim B'(S_1)\cdot V
	= S_1 e^{L(S_1)}
	\\ &\fgt
	e^{L(S_1)} > e^{L(S_1)}-e^{L(S_2)} > 0.
\end{align*}
So
$$
	B(S_2) - B(S_1) \fgt e^{L(S_1)}-e^{L(S_2)}
	= |\fa(S_2) - \fa(S_1)| .
$$

\textit{Case 2.}  $S_2-S_1 \fst \G_{N}\circ S_1$.
Now $S_2 - S_1 \fst \G_{N-1}\circ S_1$, so by
Proposition~\ref{tterm} we have
\begin{align*}
	B(S_2)-B(S_1) &\sim B'(S_1)\cdot(S_2-S_1),\\
	L(S_2)-L(S_1) &\sim L'(S_1)\cdot(S_2-S_1) .
\end{align*}
But $L \in \T_{\G_N\cup\{\log\}}$, so
$L' \in \T_N$, so $\mag(1/L') \in \G_N$, and thus
$S_2-S_1 \fst 1/L'(S_1)$ so
$$
	U := L(S_1)-L(S_2) \sim L'(S_1)\cdot(S_1-S_2) \fst 1 .
$$
Expand using the Maclaurin series for $e^z$:
\begin{align*}
	\fa(S_1)-\fa(S_2)
	&= e^{L(S_1)}(1-e^{-U})
	= e^{L(S_1)}(U+\o(U))
	\\ &\sim
	-e^{L(S_1)}L'(S_1)\cdot (S_2-S_1)
	= -\fa'(S_1)\cdot (S_2-S_1)
	\\ &\fst B'(S_1)\cdot (S_2-S_1)
	\sim B(S_2)-B(S_1).
\end{align*}
This completes the proof.
\end{proof}

\begin{lem}\label{ppg}
Let $N \in \N$.
Suppose $\G_N$ satisfies $\CC$ and $\DD$.  Then
$\G_{N+1}$ satisfies $\DD$.
\end{lem}
\begin{proof}
Since $\G_N$ satisfies $\CC$ and $\DD$, we have:
$\G_N \cup \{\log\}$ satisfies $\CC$ by Lemma~\ref{ppl}
and $\DD$ by Lemma~\ref{pph}; and
$\T_{\G_N\cup\{\log\}}$ satisfies $\CC$ by Lemma~\ref{ppc}
and $\DD$ by Lemma~\ref{ppb}.

Let $\fa, \fb \in \G_{N+1}$ with $\fa' \fst \fb'$
and let $S_1, S_2 \in \P$ with $S_1 < S_2$.
Since $\fb=1$ is impossible
and $\fa=1$ is easy, assume they are not $1$; so
$\fa \fst \fb$.
Note $\log \fb \in \T_{\G_N\cup\{\log\}}$ is purely large
and nonzero, hence large.  

Let $\m = \fa/\fb$ so that $\m \fst 1$,
and thus $\m(S_1) \fst 1$, $\m(S_2) \fst 1$.

I claim that
\begin{equation*}
	\fb(S_1)\,\frac{\m(S_2)-\m(S_1)}{\fb(S_2)-\fb(S_1)} \fst 1 .
\tag{2}\label{eq:ratio}
\end{equation*}
We will prove this in cases.

\textit{Case 1: }$\fb(S_1) \fgt \fb(S_2)$.  Then
$\fb(S_1)-\fb(S_2) \sim \fb(S_1)$, so
$$
	\fb(S_1)\,\frac{\m(S_2)-\m(S_1)}{\fb(S_2)-\fb(S_1)}
	\sim \m(S_1) - \m(S_2) \fst 1,
$$
as claimed.

\textit{Case 2: }$\fb(S_1) \fsteq \fb(S_2)$.  If
$\fb(S_2) > \fb(S_1)$, then apply
Lemma~\ref{ppp} [to $\m \fst 1 \fst \log \fb$] to get
\begin{align*}
	\fb(S_1)\,\frac{\m(S_2)-\m(S_1)}{\fb(S_2)-\fb(S_1)}
	&\fst
	\fb(S_1)\,\frac{\log \fb(S_2)-\log \fb(S_1)}{\fb(S_2)-\fb(S_1)}
	\\ &=
	\fb(S_1)\,\frac{\log\big(\fb(S_2)/\fb(S_1)\big)}{\fb(S_2)-\fb(S_1)}
	\\ &<
	\fb(S_1)\,\frac{\big(\fb(S_2)/\fb(S_1)\big)-1}{\fb(S_2)-\fb(S_1)}
	= 1 .
\end{align*}
On the other hand, if
$\fb(S_2) < \fb(S_1)$, then again apply
Lemma~\ref{ppp} [to $\m \fst 1 \fst \log \fb$] to get
\begin{align*}
	\fb(S_1)\,\frac{\m(S_1)-\m(S_2)}{\fb(S_1)-\fb(S_2)}
	&\fst
	\fb(S_1)\,\frac{\log \fb(S_1)-\log \fb(S_2)}{\fb(S_1)-\fb(S_2)}
	\\ &=
	\fb(S_1)\,\frac{\log\big(\fb(S_1)/\fb(S_2)\big)}{\fb(S_1)-\fb(S_2)}
	\\ &<
	\fb(S_1)\,\frac{\big(\fb(S_1)/\fb(S_2)\big)-1}{\fb(S_1)-\fb(S_2)}
	= \frac{\fb(S_1)}{\fb(S_2)} \fsteq 1 .
\end{align*}
So in both cases, we have established (\ref{eq:ratio}).

Now compute
\begin{align*}
	\fa(S_2) - \fa(S_1)
	&= \fb(S_2)\m(S_2) - \fb(S_1)\m(S_1)
	\\ &=
	\big(\fb(S_2)-\fb(S_1)\big)\,\left(\m(S_2)
	+ \fb(S_1)\,\frac{\m(S_2)-\m(S_1)}{\fb(S_2)-\fb(S_1)}\right)
	\\ &\fst \fb(S_2)-\fb(S_1) .
\end{align*}
The final step uses (\ref{eq:ratio}) together with $\m(S_2) \fst 1$.
\end{proof}

\begin{pr}\label{ppi}
$\T_\bullet = \R\lbb\G_\bullet\rbb$ satisfies $\CC$ and $\DD$.
\end{pr}
\begin{proof}

By Lemmas~\ref{ppd} and~\ref{ppe} $\G_0$ satisfies $\CC$ and $\DD$.
Applying Lemmas~\ref{ppf} and~\ref{ppg} inductively,
we conclude that $\G_N$ satisfies $\CC$ and $\DD$ for
all $N \in \N$.  And therefore $\G_\bullet = \bigcup_N \G_N$
satisfies $\CC$ and $\DD$ by Remark~\ref{ppa}.  Finally $\T_\bullet$
satisfies $\CC$ and $\DD$ by Lemmas \ref{ppb} and~\ref{ppc}.
\end{proof}

\begin{pr}
Let $\RR \subseteq \T$ and define
$\widetilde{\RR} := \SET{T\circ\log}{T \in \RR}$.
If $\RR$ satisfies $\CC$, then
$\widetilde{\RR}$ satisfies $\CC$.
If $\RR$ satisfies $\DD$, then
$\widetilde{\RR}$ satisfies $\DD$.
\end{pr}
\begin{proof}
Assume $\RR$ satisfies $\CC$.  Let $Q \in \widetilde{\RR}$, so that
$Q = T \circ \log$ with $T \in \RR$.  Note
$Q' = (T'\circ \log)/x$, so that $T'$ and $Q'$ have the same sign.
Let $S_1, S_2 \in \P$ with $S_1 < S_2$.  Then
$\log(S_1), \log(S_2) \in \P$ with $\log(S_1) < \log(S_2)$.
Now if $T'>0$, then applying property $\CC$ of $\RR$ to
$\log(S_1)$ and $\log(S_2)$, we get $T(\log(S_1))<T(\log(S_2))$.
That is: $Q(S_1) < Q(S_2)$.  The case $T'=0$ and $T'<0$
are similar.

The proof for $\DD$ is done in the same way.
\end{proof}

\begin{thm}\label{posderivthm}
The whole transline $\T$ satisfies $\CC$ and $\DD$.
\end{thm}

\section{Further Transseries}\label{furthersection}
Suppose we allow well-based transseries, but do not end in $\omega$
steps.  Begin as in Definition~\ref{beyond}.  Write
$\WW_{\omega} = \WW_{\bullet,\bullet}$, where $\omega$ is
the first infinite ordinal.  Then proceed
by transfinite recursion:
If $\alpha$ is an ordinal and $\G_\alpha$ has been defined, let
$\T_\alpha = \R[[\G_\alpha]]$ and 
$\WW_{\alpha+1} = \SET{e^L}{L \in \T_{\alpha} \text{ is purely large}}$.
If $\lambda$ is a limit ordinal and $\WW_{\alpha}$
have been defined for all $\alpha < \lambda$, let
$$
	\WW_{\lambda} = \bigcup_{\alpha<\lambda} \WW_{\alpha} .
$$
See \cite[\S2.3.4]{schmeling}.  Does it exist elsewhere, as well?

Call the elements of $\WW_\alpha$ \Def{Schmeling transmonomials}
and the elements of $\T_\alpha$ \Def{Schmeling transseries}.
This will allow such transseries as
$$
	H := \log x + \log\log x + \log\log\log x + \cdots
$$
and such monomials as
$$
	G := e^{-H} = \frac{1}{x\log x \log\log x \log\log\log x \cdots}\,.
$$
(In the notation of \cite[\S2.3]{schmeling},
$H \in \mathbb{L}$ and $G \in \mathbb{L}_\mathrm{exp}$.)
This $G$ is interesting (as those who have thought about convergence
and divergence of series will know) because: for actual
transseries $T$, we have $\int T \fgt 1$ if and only if $T \fgt G$.
That is, for $S \in \T$ we have: if $S \fgt 1$ then $S' \fgt G$;
if $S \fst 1$ then $S' \fst G$.

So what happens if we attempt to investigate $\int G$ if possible?
It seems that there is no Schmeling transseries $S$ with $S'=G$.

\subsection*{Iterated Log of Iterated Exp}
A Usenet sci.math discussion in July, 2009, suggested
investigation of growth rate of a function $Y$ with
$Y = \log(Y(e^{ax}))$ for a fixed constant $a$ (there it
was $\log 3$).  This $Y$ should be a limit of the sequence:
\begin{align*}
	Y_0 &= x ,
	\\
	Y_1 &= \log(e^{ax}) ,
	\\
	Y_2 &= \log\left(\log\left(e^{ae^{ax}}\right)\right) ,
	\\
	Y_3 &= \log\left(\log\left(\log\left(e^{\displaystyle
	ae^{ae^{ax}}}\right)\right)\right) ,
\end{align*}
and so on.
Iteration of transseries suggests a solution $Y$ not of finite height.
It seems $Y$ should begin
\begin{align*}
Y &= a x + \log(a) + \frac{\log(a)}{a}\, e^{-a x}
- \frac{1}{2}\,\frac{\log(a)^2}{a^2}\, e^{-2 a x} 
+ \frac{1}{3}\,\frac{\log(a)^3}{a^3}\, e^{-3 a x}
\\ &\qquad
- \frac{1}{4}\,\frac{\log(a)^4}{a^4}\, e^{-4 a x}
+ \frac{1}{5}\,\frac{\log(a)^5}{a^5}\, e^{-5 a x}
- \frac{1}{6}\,\frac{\log(a)^6}{a^6}\, e^{-6 a x}
+ \cdots
\end{align*}
and so on; order-type $\omega$.
Writing $\mu_1$ for $e^{-a x}$, these terms have
coefficient times powers of $\mu_1$.
Beyond all of those, we have terms involving
$\mu_2 = \exp(-a \exp(a x))$, beginning
\begin{align*}
\mu_2 \;& \big(
\log(a)\mu_1
-\log(a)^2\mu_1^2
+ \log(a)^3\mu_1^3
- \log(a)^4\mu_1^4 
+ \log(a)^5\mu_1^5 
+ \cdots\big)
\\
+ \mu_2^2 \;& \Big(
- \frac{\log(a)^2}{2}\mu_1
+ \frac{\log(a)^3-\log(a)^2}{2}\mu_1^2
+ \frac{2\log(a)^3-\log(a)^4}{2}\mu_1^3 
\\
&\qquad + \frac{\log(a)^5-3\log(a)^4}{2}\mu_1^4 
+ \frac{4\log(a)^5-\log(a)^6}{2}\mu_1^5
+\cdots\Big) +\cdots
\end{align*}
Order-type $\omega^2$.
Beyond all those we have terms involving
$\mu_3 = \exp(-a\exp(a\exp(a x)))$; order-type $\omega^3$.
And so on with $\mu_k$ of height $k$ for $k \in \N$.

\subsection*{Surreal Numbers}
If this extension for well-based transseries is continued through
all the ordinals, the result is a large (proper class)
real-closed ordered field.
With additional operations.  J.~H. Conway's system
of \Def{surreal numbers} \cite{conway}
is also a large (proper class) real-closed
ordered field, with additional
operations.  Any ordered field (with a set of elements, not a
proper class) can be embedded in either of these.  We can build
recursively a correspondence between the well-based transseries
and the surreal numbers.  But involving many arbitrary choices.

\cite[p.~16]{hoeven}
Is there a \textit{canonical} correspondence, not only preserving
the ordered field structure, but also some of the additional
operations?  Or is there a canonical embedding of one
into the other?  Perhaps we need to take the recursive way in which
one of these systems is built up and find a natural
way to imitate it in the other system.

Reals should correspond to reals.  The transseries $x$ should
correspond to the surreal number $\omega$.  But there are
still many more details not determined just by these.


\begin{thebibliography}{99}
\bibitem{asch}
M. Aschenbrenner, L. van den Dries,
Asymptotic differential algebra.
In~\cite{proc}, pp.~49--85

\bibitem{conway}
J. H. Conway,
\textit{On numbers and games.}
Second edition. A K Peters, Natick, MA, 2001

\bibitem{costintop}
O. Costin, Topological construction of transseries and introduction to generalized Borel summability.
In~\cite{proc}, pp.~137--175

\bibitem{costinglobal}
O. Costin, Global reconstruction of analytic functions
from local expansions and a new general method of converting
sums into integrals. preprint, 2007.\hfill\break
\texttt{http://arxiv.org/abs/math/0612121}

\bibitem{costinasymptotics}
O. Costin,
\textit{Asymptotics and Borel Summability.}
CRC Press, London, 2009

\bibitem{proc}
O. Costin, M. D. Kruskal, A. Macintyre (eds.),
\emph{Analyzable Functions and Applications}
(\emph{Contemp. Math.} \textbf{373}).
Amer. Math. Soc., Providence RI, 2005

\bibitem{DMM}
L. van den Dries, A. Macintyre, D. Marker,
Logarithmic-exponential series.
\emph{Annals of Pure and Applied Logic} \textbf{111} (2001) 61--113

\bibitem{edgar}
G.~Edgar,
Transseries for beginners. preprint, 2009.\hfill\break
\texttt{http://arxiv.org/abs/0801.4877} or\hfill\break
\texttt{http://www.math.ohio-state.edu/\hbox{$\sim$}edgar/preprints/trans\_begin/}

\bibitem{edgarw}
G.~Edgar,
Transseries: ratios, grids, and witnesses. \textit{forthcoming}\hfill\break
\texttt{http://www.math.ohio-state.edu/\hbox{$\sim$}edgar/preprints/trans\_wit/}

\bibitem{edgarfi}
G.~Edgar,
Fractional iteration of series and transseries.  preprint, 2009.\hfill\break
\texttt{http://www.math.ohio-state.edu/\hbox{$\sim$}edgar/preprints/trans\_frac/}

\bibitem{higman}
G. Higman, Ordering by divisibility in abstract algebras.
\textit{Proc. London Math. Soc.} \textbf{2} (1952) 326--336

\bibitem{hoevenop}
J. van der Hoeven,
Operators on generalized power series.
\emph{Illinois J. Math.} \textbf{45} (2001) 1161--1190

\bibitem{hoeven}
J. van der Hoeven,
\emph{Transseries and Real Differential Algebra}
(\emph{Lecture Notes in Mathematics} \textbf{1888}).
Springer, New York, 2006

\bibitem{hoevenpre}
J. van der Hoeven, Transserial Hardy fields. preprint, 2006

\bibitem{kuhlmann}
S. Kuhlmann, \emph{Ordered Exponential Fields.}
American Mathematical Society, Providence, RI, 2000

\bibitem{schein}
S. Scheinberg, Power series in one variable.
\emph{J. Math. Anal. Appl.} \textbf{31} (1970) 321--333


\bibitem{schmeling}
M.~C. Schmeling, {\em Corps de transs\'eries.\/}
Ph.D. thesis, Universit\'e Paris VII, 2001

\end{thebibliography}
\end{document}